\documentclass[11pt]{amsart}

\usepackage{amssymb}
\usepackage{amsthm}
\usepackage{mathtools}
\usepackage{color}
\usepackage{amsmath}
\usepackage{graphicx}
\usepackage[T1]{fontenc}
\usepackage[english]{babel}
\usepackage{enumerate}
\usepackage{babel}
\usepackage[left=3cm,right=3cm,top=3cm,bottom=3cm]{geometry}
\usepackage[pdftex,breaklinks,colorlinks,citecolor=blue,urlcolor=blue]{hyperref}
\usepackage{verbatim}
\numberwithin{equation}{section}
\newtheorem{theorem}{Theorem}[section]

\newtheorem{Prop}[theorem]{Proposition}
\newtheorem{theo}{Theorem}[section] 
\newtheorem{lem}[theo]{Lemma}
\newtheorem{corll}[theo]{Corollary}
\newtheorem{rem}[theo]{Remark}
\newtheorem{clm}[theo]{Claim}
\newtheorem{prop}[theo]{Proposition}
\newtheorem{defi}[theo]{Definition}

\def \R {\mbox{I\hspace{-.15em}R}}
\def \C {\mbox{\,l\hspace{-.47em}C}}
\def \N {\mbox{I\hspace{-.15em}N}}

\DeclareMathOperator{\re}{Re}
\DeclareMathOperator{\im}{Im}

\title[3D  NLS outside a strictly convex obstacle]{Construction of solitary wave solution of the nonlinear focusing schr\"odinger equation outside a strictly convex obstacle for the $L^2$-supercritical case}
\author{Landoulsi Oussama}
\email[]{landoulsi@univ-paris13.fr}
\address{LAGA, Universit\'e Paris 13, Sorbonnee Paris Cit\'e, UMR 7539}

\begin{document}

\maketitle

\begin{abstract}

We consider the focusing $L^2$-supercritical Schr\"odinger equation in the exterior of a smooth, compact, strictly convex obstacle $\Theta \subset \R^3$. We construct a solution behaving asymptotically as a solitary waves on $\R^3$, as large time. When the velocity of the solitary wave is high, the existence of such a solution can be proved by a classical fixed point argument. To construct solutions with arbitrary nonzero velocity, we use a compactness argument similar to the one that was introduced by F.Merle in 1990 to construct solution of NLS blowing up at several blow-up point together with a topological argument using Brouwer's theorem to control the unstable direction of the linearized operator at soliton. These solutions are arbitrarily close to the scattering threshold given by a previous work of  R.\,Killip, M.\,Visan and X.\,Zhang which is the same as the one on whole Euclidean space.   

\end{abstract}

\tableofcontents
\newpage
\section{Introduction}
We consider the focusing nonlinear Schr\"odinger equation in the exterior of a smooth compact strictly convex obstacle $\Theta \subset \R^3$ with Dirichlet boundary conditions: \\
	\medskip
	\begin{multline}
	\label{NLS} 
		(\text{NLS}{_\Omega})
			\left\{
			\begin{array}{rrrrrrrr}
			\begin{aligned}
	i\partial_tu+\Delta_{\Omega} u &= -|u|^{p-1}u  \qquad  &\forall (t,x)\in [T_0,+\infty)\times\Omega ,\\ u(T_0,x) &=u_0(x)   &\forall x \in \Omega  , \\
	u(t,x)&=0  &\forall(t,x)\in[T_0,+\infty)\times \partial\Omega .
		\end{aligned}
			\end{array}
		\right. \end{multline}  
	\medskip	
     Where $\Omega=\R^3 \setminus \Theta$, $\Delta_{\Omega}$ is the Dirichlet Laplace operator on $\Omega$ which is a self-adjoint operator with form domain $H^1_0(\Omega)$, $\partial_t$ is the derivative with respect to the time variable and $T_0 \in \R$ is the initial time. Here $u$ is a complex-valued function,
\begin{align*}
 u:  [T_0,+\infty)&\times  \Omega \longrightarrow \C \\
 	(t &, x)  \longmapsto  u(t,x) .
 \end{align*}
    We take the initial data $u_0 \in H^1_0(\Omega)$.  \\ 
 
The local Cauchy problem for (NLS$_\Omega$) in $H^1_0(\Omega)$ was studied in several articles. For $1<p<5$, L.\,Vega and F.\,Planchon proved that (NLS) equation in the exterior of a non-trapping domain in $\R^3$ is locally well-posed, see \cite{PlVe09}. After that,  F.\,Planchon and O.\,Ivanovici extended the result to the quintic Schr\"odinger equation outside a non-trapping domain see \cite{MR2683754}.\\

 The solutions of the (NLS$_{\Omega}$) satisfy the mass and energy conservation laws: 
 \begin{align*}
 M(u(t)) &:= \int_{\Omega} |u(t,x)|^2 dx = M(u_0) . \\
E(u(t)) &:= \frac{1}{2} \int_{\Omega} |\nabla u(t,x)|^2 dx - \frac{1}{p+1} \int_{\Omega}\left|u(t,x)\right|^{p+1}dx=E(u_0).  
\end{align*} 
 
 Furthermore, the (NLS) equation posed on the whole Euclidean space $\R^3$ is invariant by the scaling transformation, that is, 
 \begin{equation*}
 u(t,x) \longmapsto \lambda^{\frac{2}{p-1}}u(\lambda x,\lambda^2 t) \;, \; \text{ for }  \lambda >0.
 \end{equation*}
 
 This scaling identifies the critical Sobolev space $\dot{H}^{s_c}_{x}$, where the critical regularity $s_c$ is given by $s_c:=\frac{3}{2} - \frac{2}{p-1}$. The case when $ s_c=0$ is referred to as mass-critical or $L^2$-critical and the case when $s_c=1$ is called energy-critical or $H^1$-critical. \\  

Throughout this paper, we will take $ \frac{7}{3}<p<5$. Since the presence of the obstacle does not change
the intrinsic dimensionality of the problem, so we may regard (NLS$_{\Omega}$) equation as being $H^1(\Omega)$-subcritical and $L^2(\Omega)$-supercritical.  \\

Consider solitary waves solution of  (NLS$_{\Omega}$), with $\Omega=\R^3$, that is $u(t,x)=e^{i t \omega }Q_{\omega}(x)$ \\ 
where $ Q_\omega $ is a solution of the nonlinear elliptic equation: 
 \begin{equation}
 \label{eq_Q}
 \begin{cases}
 -\Delta Q_\omega + \omega \, Q_\omega= \left| Q_\omega\right|^{p-1} Q_\omega ,   \\
 \;  Q_\omega \in H^1(\R^3).
 \end{cases}
 \end{equation} 
 
This elliptic equation admits solutions if and only if $\omega > 0$. In this paper, we will denote by $ Q_\omega $ the ground state which is the unique radial positive solution of \eqref{eq_Q}. We recall that $Q_{\omega}$ is smooth and exponentially decaying at infinity and characterized as the unique minimizer for the Gagliardo-Nirenberg inequality up to scaling, space translation and phase shift, see \cite{Kwong89}. \\

The (NLS) equation posed on the whole Euclidean space $\R^3$, also enjoys Galilean invariance. If $u(t,x)$ is solution, then $u(t,x-vt) \, e^{i(\frac{x.v}{2} -   \frac{ | v |^2   }{4} t ) } $ is also a solution, for $v \in \R^3$.\\

Applying a Galilean transform to the solution $e^{i  t \omega } Q_{\omega}(x)$ of the (NLS) on $\R^3$, we obtain a soliton solution, moving on the line $x=t v$ with velocity $v \in \R^3:$ 
\begin{equation}
\label{soliton}
u(t,x) = e^{i(\frac{1}{2}(x.v)-\frac{1}{4} \left| v\right|^2 t + t \, \omega )}Q_{\omega}(x-t \,v).
\end{equation} 
 
The soliton \eqref{soliton} is a global solution of the focusing nonlinear Schr\"odinger equation (NLS) posed on the whole space, but is not a solution of (NLS$_\Omega$). Our goal is to construct solitary waves of the (NLS$_{\Omega}$) satisfying Dirichlet boundary conditions and behaving asymptotically as the preceding solitary waves $e^{i(\frac{1}{2}(x.v)-\frac{1}{4} \left| v\right|^2 t + t \, \omega )}Q_{\omega}(x-t \,v)  $,  as $ t \longrightarrow +\infty $. \\
The main result of this paper is the following. 
\begin{theo}
\label{theorem-pranp}
Assume $\frac{7}{3}<p<5$.\\ 	
Let $\Psi$ be a $C^{\infty}$ function such that: $\left\{ \begin{array}{rr}
\Psi=0 & \; \text{near} \;\, \Theta , \\
\Psi=1 & \; if \;  \left|x\right| \gg 1 .
 \end{array}\right.$ \\
Let $v \in \R^3 \backslash \{ 0 \}$ be the velocity, $\omega>0$. Then there exists $\delta > 0$, $T_0>0$  and  a function $r_\omega$ defined on $[T_0 , + \infty ) \times \Omega $  satisfying $$\left\|r_\omega (t) \right\|_{H^1 _0(\Omega)} \leq  e^{-\delta \sqrt \omega |v| t }   \qquad \forall t \in [T_0,+\infty), $$ 
	such that, \begin{equation*}
	    \label{solution-u}
	 u(t,x)=e^{i(\frac{1}{2}(x.v)-\frac{1}{4}|v|^2 t +t \, \omega)} Q_\omega(x-tv) \Psi(x) + r_\omega(t,x), \quad \forall (t,x) \in [T_0,+\infty) \times \Omega, \end{equation*}
	is a solution of \rm{(NLS$_\Omega$).}
\end{theo}
\begin{rem} Theorem \ref{theorem-pranp} can be generalized for any dimension $d \geq 3$. Moreover, this result can be extended to the subcritical case $1<p<\frac{7}{3}$ which are easier to prove due to the stability of solitons.
\end{rem}
\begin{rem}
The restriction to a strictly convex obstacle is purely technical. In section 2, we will need that the (NLS$_\Omega$) equation is well posed on $H^s(\Omega)$, for some  $s \in [s_p,1[, $ with $s_p= \frac{3}{2}-\frac{3}{p+1}$ (Cf. Lemma  \ref{wellposed}), for that we need to use a Strichartz estimate from \cite{Ivanovici10} (Cf.~Theorem \ref{strichartz}) and some fractional rules given by \cite{killip2015riesz} for strictly convex obstacle (Cf. Proposition \ref{Prop-fractional-rule}). Because of this, we shall suppose that the obstacle $\Theta$ is strictly convex. \\ 
\end{rem}

In \cite{DuHoRo08}, T.\,Duyckaerts, J.\,Holmer and S.\,Roudenko have studied the behavior (i.e scattering and global existence) of the solutions of the focusing cubic (i.e p=3) nonlinear Schr\"odinger equation on $\R^3$, whenever the initial data satisfies a smallness criterion given by the ground state threshold. The criterion is expressed in terms of the scale-invariant quantities $\left\|u_0\right\|_{L^2}\left\| \nabla u_0 \right\|_{L^2}$ and $M(u)E(u).$ This result was later extended to arbitrary space dimensions and focusing mass-supercritical power nonlinearities by T.\,Cazenave, J.\, Xie and D.\, Fang, see \cite{FaXiCa11} and by~C.\,Guevara in \cite{Guevara14}.

\begin{theorem}[\cite{FaXiCa11},\cite{DuHoRo08},\cite{Guevara14}]
\label{thresholdOutObstacle}
Let $s=\frac{3}{2}-\frac{2}{p-1}$ and  $\frac{7}{3}<p<5.$ Let $u_0 \in H^1(\R^3)$ satisfy
\begin{align}
    \left\| u_0 \right\|^{1-s}_{L^2\footnotesize{\left(\R^3 \right)}}\left\| \nabla u_0 \right\|^{s}_{L^2\footnotesize{\left(\R^3 \right)}}  & < \left\| Q  \right\|^{1-s}_{L^2\footnotesize{\left(\R^3 \right)}} \left\| \nabla Q \right\|^{s}_{L^2\footnotesize{\left(\R^3 \right)}} , \\
    M(u_0)^{1-s} E(u_0)^s & < M(Q)^{1-s} E(Q)^{s}.
\end{align}
Then u scatters in $H^1(\R^3)$.

\end{theorem}
 
Theorem \ref{thresholdOutObstacle} remains true for (NLS$_\Omega$) in the exterior of a strictly convex obstacle in three dimension. Indeed, R.\,Killip, M.\,Visan and X.\,Zhang had proved in \cite{MR3483844} that the threshold for global existence and scattering is the same as for the cubic equation on $\R^3$. Moreover, K.\,Yang extended this result for $\frac{7}{3}<p<5$, see \cite{Kai17}. \\

The solitary waves constructed in the main Theorem \ref{theorem-pranp} prove the optimality of the threshold for scattering given in \cite[Theorem 1.3] {Kai17}. Indeed, the solution $u$ of (NLS$_\Omega$) is global, does not scatter for positive time direction and we have\begin{equation}
\label{E(u)}
E(u)= \frac{\left| v \right|^2}{8} \int \left| Q \right|^2 + E(Q) \, .
\end{equation}

Since, the velocity $v$ can be taken arbitrary small, we have proved that for all $ \varepsilon >0 $ there exists a solution $u_\varepsilon$ of (NLS$_\Omega$) which is global and does not scatter for positive time such that $$\displaystyle M(u_\varepsilon)=M(Q)\, ,   \quad  \sup_{ t \geq T_0}  \left\| \nabla u_\varepsilon(t) \right\|_{L^2(\Omega)}^s < \left\|  \nabla Q \right\|_{L^2\footnotesize{\left(\R^3 \right)}}^s + \varepsilon \; $$ and  $$ E(u_\varepsilon)^s < E(Q)^s  + \varepsilon\,.$$

The proof of Theorem \ref{theorem-pranp} relies on a compactness argument that uses the structure of the linearized operator around the ground state soliton. If the velocity $v$ is large enough, we can use a simple fixed point theorem to construct a soliton solution of (NLS$_\Omega$).
\begin{theo}
\label{theoremWithVelocityGrand}
Assume $ 2 \leq p <5$. \\Let $\Omega=\R^3 \setminus \Theta $ where $\Theta$ is any smooth compact obstacle and $Q_\omega$ be any solution~of \eqref{eq_Q}.
\\Let $\Psi$ be a $C^{\infty}$ function such that: $\left\{ \begin{array}{rr}
\Psi=0 & \; near \; \Theta ,\\
\Psi=1 & \; if \;  \left|x\right| \gg 1 .
 \end{array}\right.$ \\
Let $\omega,T_0>0$. Then there exists $V_0:=V_0(\omega) \gg 1 $ with the following property. Let $v \in \R^3$ be the velocity such that $\left| v \right| > V_0.$ \\ 

Then there exists $\delta>0$ and a functions $r_\omega$ defined on $[T_0,+\infty) \times \Omega $ satisfying
$$ \hspace{-2cm} \forall t \in [T_0,+\infty )  \qquad  \left\| r_\omega (t) \right\|_{H^2 \cap H^1_0(\Omega)} \leq C_{\omega}\left|v\right|^{3} e^{-\delta \sqrt{\omega} \left| v\right| t  } ,  $$
	such that   $  \; u(t,x)=e^{i(\frac{1}{2}(x.v)-\frac{1}{4}|v|^2 t +t \, \omega)} Q_\omega(x-tv) \Psi(x) + r_\omega(t,x), \quad \forall (t,x) \in [T_0,+\infty) \times \Omega, $\\
	is a solution of \rm{(NLS$_\Omega$).}
\end{theo}
Unlike in Theorem \ref{theorem-pranp}, $Q_\omega$ is any  solution of the nonlinear elliptic equation \eqref{eq_Q} (not necessarily the ground state) and $\Theta \subset \R^3 $ does not have to be convex, which makes Theorem~\ref{theorem-pranp}  more general for high velocity. However, we can see in \eqref{E(u)} that the choice of high velocity does not allow us to use Theorem \ref{theoremWithVelocityGrand} to show the optimality of the threshold for scattering in \cite{MR3483844} and \cite{Kai17}. Let us mention that, this result can be extended for any dimension $d \geq 3$. We will give the proof of the Theorem \ref{theoremWithVelocityGrand} for the cubic case $p=3$. The proof for general $p\in [2,5)$ is very similar, see Remark \ref{proof-for-any-p}. \\

Let us mention that apart from the works of R.\,Killip, M.\,Visan and X.\,Zhang and K.\,Yang cited above and in \cite{KiVisZha16}, (NLS$_{\Omega}$) outside obstacle was also studied by N.\,Burq, P.\,Gerard and N.\,Tzvetkov in \cite{BuGeTz04a} and F.\,Abou Shakra in \cite{Farah15}.
Let us also mention the recent works on dispersive estimates outside one or several strictly convex obstacles of O.\,Ivanovivi and G.\,Lebeau in \cite{IvanoLebo17} and D.\,Lafontaine in \cite{lafontaine2017strichartz}, \cite{lafontaine2018strichartz}.  \\ 

We end this section by giving sketch of the proofs of the two theorem. \\
\textbf{Sketch of the proof of  Theorem $\ref{theorem-pranp}$.} \\ 
The structure of the proof is similar to the one for construction of multi-soliton for (NLS) on $\R^d$ in the subcritical case in  \cite{MR2271697} with an additional argument coming from  \cite{MR2815738} which allows to handle the supercritical character of the non-linearity. The compactness argument used in the present paper is similar to the main argument used in  \cite{MR2271697},\cite{MR2815738}, and \cite{Merle90}. \\ 

Note that, even though we use some similar arguments, a large part of the proof of Theorem~\ref{theorem-pranp} is different.  It is because of the presence of the obstacle $\Theta$ which makes the calculations more complicated.  \\

Recall that the soliton $Q_{\omega}(x-t \,v)  e^{i(\frac{1}{2}(x.v)-\frac{1}{4} \left| v\right|^2 t + t \, \omega )}$ is an exact solution of the (NLS) on the whole space $\R^3$. So, the proof consists in the construction of a smooth correction $ r_\omega(t,x)$ with some uniform estimates, such that $  R(t,x)$ $+$ $r_\omega(t,x)$ is a solution of the equation  (NLS$_\Omega$) where 
$R(t,x) = e^{i(\frac{1}{2}(x.v)-\frac{1}{4} \left| v\right|^2 t + t \, \omega )}Q_{\omega}(x-t \,v) \,\Psi(x). $\\ 

The paper is organised as follows. In \S \ref{Properties of the ground state}, we give a review of some properties of the ground state $Q$. In \S \ref{Spectral theory of the linearized operator}, we recall some spectral properties of the linearized Schr\"odinger operator around the soliton $e^{it}Q$. That is, \\

In the subcritical case, Cazenave and Lions \cite{CaLi82}, Weinstein \cite{MR820338} proved that the solitary waves are stable when $1<p<\frac{7}{3}$, which means that the nonlinearity has a $L^2$-subcritical growth. \\ 
From \cite{MR820338}, there exits $\lambda >0 $ such that for any real-valued function $h \in H^1, $
$$\displaystyle (h,Q_\omega),(h,\nabla Q_\omega)=0 \Longrightarrow \displaystyle\int \{\left| \nabla h\right|^2 + \omega \left| h\right|^2 - p Q_\omega^{p-1} \left|Q_\omega \right| ^2  \} \geq \lambda \left\| h \right\|_{H^1}^2 .
$$
In \cite{MR2271697}, the authors use some modulation in the scaling, phase and translation parameters, to~control these two direction.\\

In the supercritical case, it is well known that the soliton is unstable, see \cite{MR901236}. Indeed, for~$\frac{7}{3}<p<5$, there exists two eigenfunctions of the linearized operator around the ground state Q, constructed by Weinstein \cite{Weinstein85}, Schlag \cite{Sc06}, Grillakis~\cite{MR1040143} and denoted by $\mathcal{Y}^{\pm}$. Thus, the above property of the linearized operator does not hold, but a effective coercivity property can be expressed in term of the eigenfunctions $\mathcal{Y}^{\pm}$, see Lemma \ref{spectalprop}. \\

In \S \ref{compactness argumenttt}, we suppose that there exists a solution $u_n$ of (NLS$_\Omega$) for $t\in[T_0,T_n]$ that satisfies some uniform estimate with initial data $u_n(T_n)$ and $T_n$ is an increasing sequence of times. Then by compactness argument we construct a solution~$u$ of (NLS$_\Omega$) for $[T_0,+\infty)$, with initial data $u(T_0)$ and $T_0>0$, which concludes the proof of Theorem \ref{theorem-pranp}. \\

In Section \ref{Proof of the uniform estimate}, we prove the existence of the solution $u_n$ and the uniform estimate assumed in the previous section. For this, we use a modulation for large time in the phase and translation parameters in the decomposition of the solution as above to obtain some orthogonality conditions. Next, we define a maximal time interval on which hold a suitable exponential estimates of the modulation parameters, the uniform estimate used in \S \ref{compactness argumenttt} and others terms expressed in function of $\mathcal{Y}^{+}$ and $\mathcal{Y}^{-}$. In order to control these estimates, we use a bootstrap argument with the coercivity property of the linearized operator. Indeed, the linearized operator $(\mathcal{L} \;  \cdot  \, , \, \cdot) $ is positive definite up to the four directions $Q$, $\partial_x Q$ and $ \mathcal{Y}^{\pm}$, see \cite{DuMe09a} and \cite{DuRo10}. As in the subcritical case, the two directions $Q_{\omega},\nabla Q_{\omega} $ are still be controlled due to the orthogonality conditions given by the modulation. The direction $\mathcal{Y}^{+}$ is stable in some sense that can be controlled but the other one $\mathcal{Y}^{-}$ is unstable and cannot be controlled by a scaling argument, even if we introduce an extra parameters in the modulation. Thus, we have to use a topological argument to control this unstable direction and to conclude the proof of the uniform estimate on $[T_0,T_n].$   \\

\textbf{Sketch of the proof of Theorem \ref{theoremWithVelocityGrand}. \\ }
In section \ref{fixed point theorem}, we will give the proof of Theorem \ref{theoremWithVelocityGrand}. We construct a contraction mapping of a complete metric space to itself using the Duhamel formula. By fixed point theorem we prove the existence of a smooth correction $r_\omega(t,x)$ such that $u(t,x)=R(t,x)+r_{\omega}(t,x)$ is a solution of (NLS$_{\Omega}$), where $R(t,x)=e^{i(\frac{1}{2}(x.v)-\frac{1}{4} \left| v\right|^2 
t + t \, \omega )}Q_{\omega}(x-t \,v) \Psi(x)\, . $ \\ 

We have
$$(i \partial_t + \Delta)R(t,x)= -\Psi(x) \left| H(t,x) \right|^2 H(t,x) + 2 \nabla \Psi(x) \nabla H(t,x) + \Delta \Psi(x) H(t,x),$$ where $H(t,x)= e^{i(\frac{1}{2}(x.v)-\frac{1}{4} \left| v\right|^2 t + t \, \omega )}Q_{\omega}(x-t \,v) .$ \\

We look for $r_w \in C([T_0,+\infty), H^2(\Omega) \cap H^1_0(\Omega) ) $ such that   \\ 
\begin{equation}
\label{problemDeP_Omega}
    \begin{cases}
    i\partial_t r_\omega + \Delta r_\omega= -\left| R+r_\omega \right|^2 (R+r_\omega) + \Psi \left|H \right|^2 H - 2 \nabla \Psi \nabla H - \Delta \Psi H\, ,  \\
   \qquad \quad  r_\omega(t) \longrightarrow 0   \quad  t \longrightarrow +\infty \quad  in \quad H^2(\Omega) \cap H^1_0(\Omega). 
\end{cases}
\end{equation}
\\ 

We shall look for solutions of (NLS$_\Omega$) in the following space 
\begin{align*}
E=\left\{ r_\omega \in C\left([T_0,+\infty),H^2(\Omega) \cap H^1_0(\Omega)\right), \;
\left\|r_\omega \right\|_E < \infty       \right\} \, ,  \\ \left\| r_\omega \right\|_E = \sup_{ t\geq T_0} \left\{ e^{\delta \sqrt{\omega} \left| v \right| t }\left( \frac{1}{\left|v\right|^3} \left\|r_\omega \right\|_{H^2(\Omega)} + \left\| r_\omega \right\|_{L^2(\Omega)} \right) \right\} .
\end{align*}
Let \begin{align*}
    \Phi:(B_E,d_E) &\longrightarrow (B_E,d_E)  \\
    r_\omega \longmapsto  & \Phi(r_\omega)= -i \int_{t}^{+\infty} S(t-\tau) \left( \left| R+r_\omega \right|^2 (R+r_\omega) + \Psi \left|H \right|^2 H - 2 \nabla \Psi \nabla H - \Delta \Psi H \right)d\tau. 
\end{align*}
Where S(t) is the unitary group of the linear Schr\"odinger equation on $\Omega$ with Dirichlet boundary conditions. \\ 

$ B_E:=B_E(0 , 1)=\{ h \in E, \;  \left\| h\right\|_E \leq 1 \}$ and $\displaystyle d_E(h,g) =  \left\|h-g\right\|_{E}$. \\ One can check that $(B_E,d_E)$ is a complete metric space. \\

  Our goal is to solve the integral formulation of \eqref{problemDeP_Omega} by the contraction mapping principle. Using the high velocity assumption, we prove that $\Phi $ is stable on $B_E$ and it is a contraction mapping. Thus, by the fixed point theorem we conclude that there exists a unique solution~$r_\omega$ of $\eqref{problemDeP_Omega}$ on $E$.\\

  Appendix \ref{Appendix du theorem pranp} contains the proof of the coercivity property of the linearized Schr\"odinger operator, the local existence of the equation on $H^s(\Omega)$, for some $s$, the modulation for time independent function and other technical results. \\ 
 Appendix \ref{Appendix du theorem Grand vitesse}, contains the computation of some estimates used on the proof of Theorem \ref{theoremWithVelocityGrand}. \\

\textbf{Notation:  \\}
If $a$ and $b$ are two functions of t and if b is positive, we write $ a = O (b)$ when there exists a constant $C>0$ independent of $t$ such that $\left| a(t) \right| \leq C \,  b(t)$ for all $t$.  \\
For $h \in \C $, we denote $h_1=\re h $ and $h_2=\im h $. \\
Throughout this paper, $C$ denotes a positive constant independent of t, that may change from line to line and may depend on $\omega$ and $\Omega$.\\
We denote by $\left| \cdot \right|$ the $\R^d$-norm with $d=1,2,3$.\\ 
For simplicity, we will write $\Delta :=\Delta_\Omega$. \\ 
Denote by $(\cdot, \cdot )$ , the real $L^2$-scalar product,
$$ (f,g)=\re{\int f \;  \overline{g}} = \int \re g \, \re f + \int \im g \,  \im f \; .  $$

\section*{Acknowledgements} 
I would like to thank my PhD advisor Prof. Thomas Duyckaerts (LAGA) for his valuable comments and suggestions which helped improve the manuscript.

\section{Construction of the solution assuming uniform estimates}
\label{Construction of the solution}
\subsection{Properties of the ground state}
\label{Properties of the ground state}
We recall some well-known properties of the ground state and we refer the reader to \cite{Weinstein82}, \cite{Kwong89} , \cite[Appendix B] {Tao06BO} and \cite{HoRo08} for more details. 
\begin{prop}[Exponential decay of $Q$] 
\label{proporties of Q_omega}
Let $Q$ be a solution of $(\ref{eq_Q})$ with $\omega=1$ then the following properties hold: \\
	$1) $ $Q $ $\in \; W^{3,p}(\R^3) $ for every $2 \leq p < +\infty . $ In particular,
	$Q \in  C^2$ and $|D^{\beta} Q(x)| \; \longrightarrow 0 $, as~$|x| \longrightarrow \infty$, for all $|\beta| \leq 2.$\\
	$2)$ there exists $\delta>0$ such that $$ e^{\delta \left|x \right| } \left( \left|Q(x) \right| + \left| \nabla Q(x) \right| + \left| \nabla^2 Q(x)\right| \right) \in L^{\infty}(\R^{3}).$$
\end{prop}
\begin{proof}
See \cite{BeLi83b} and \cite[chapter 8] {Cazenave03BO} for the proof. 
\end{proof}
We can deduce $Q_\omega(x)$ from $Q(x):$ $Q_\omega(x)= \omega^{\frac{1}{p-1}}   Q(\sqrt{\omega}x) $. \\ 

Then, there exits $C$ and $\delta >0$ such that
\begin{equation}
    \label{eq_QQ}
     \left|Q_\omega(x) \right| + \left| \nabla Q_\omega(x) \right| + \left| \nabla^2 Q_\omega(x)\right| \leq C e^{-\delta \sqrt{\omega} \left|x\right| }.
\end{equation}
\subsection{Spectral theory of the linearized operator}
\label{Spectral theory of the linearized operator}
Consider a solution u of the nonlinear Schr\"odinger equations close to the soliton $e^{it}Q$. Let $h \in \C$ such that $h=h_1+ih_2. $  \\

We can write $u(t,x)$ as, 
$$u(t,x) = e^{i t } \left(Q(t,x) + h(t,x) \right).$$
 Note that $h$ is the solution of the following equation, 
 \begin{equation*}
\partial_t h + \mathcal{L} h = S(h), \qquad \mathcal{L}:=\begin{pmatrix}
 O & -L^{-} \\
 L^{+} & 0 
 \end{pmatrix}
\end{equation*}
where $S(h)$ contains the nonlinear terms on $h$ and  the self-adjoint operators $L^{-}$ and $L^{+}$ are defined by: \begin{equation*}
    L^+h_1 = -\Delta h_1 + h_1 - p Q^{p-1}h_1 \qquad \text{and} \qquad     L^-h_2 = - \Delta h_2 + h_2 -Q^{p-1}h_2.
\end{equation*}


In all the sequel, we assume $\frac{7}{3}<p<5$. The spectral properties of the linearized operator $\mathcal{L}$ around the ground state are well-known and we refer to \cite{Weinstein85}, \cite{Grillakis90} and \cite{Sc06} for the following Proposition. 
\begin{prop}
\label{eigen-val-funct-L}
Let $\sigma(\mathcal{L})$ be the spectrum of the operator $\mathcal{L}$ defined on $L^2(\R^3) \times  L^2(\R^3) $ and let $\sigma_{ess}(\mathcal{L})$ be its essential spectrum. Then \begin{equation*}
    \sigma_{ess}(\mathcal{L})=\left\{ i\xi : \;  \xi \in \R, \; \left| \xi \right| \geq 1   \right\},  \quad  \sigma(\mathcal{L}) \cap \R = \{-e_0,0,e_0     \} \quad \text{with } e_0>0. 
\end{equation*}
\end{prop}

Moreover, $e_0$ and $-e_0$ are simple eigenvalues of $\mathcal{L}$ with eigenfunctions $\mathcal{Y}^{+}$ and $\mathcal{Y}^{-},$
$$\mathcal{L}\mathcal{Y}^{\pm} = \pm e_0 \mathcal{Y}^{\pm},$$
and $ \overline{\mathcal{Y}^{+} }=\mathcal{Y}^{-}.$ Furthermore $\mathcal{Y}^{+}, \mathcal{Y}^{-}$ $\in  \mathcal{S}(\R^3)$, in fact there exists $\delta >0$ and $C>0$ such that  $$ \left| \mathcal{Y}^{\pm} \right| + \left| \nabla \mathcal{Y}^{\pm}\right| \leq C e^{-\delta \left| x \right| } .$$

\begin{rem}
The null-space of $L^+$ is spanned by $\partial_{x_1}Q$, $\partial_{x_2}Q$ and $\partial_{x_3}Q$ and the null-space of $L^-$ is spanned by $Q$.
\end{rem}
Moreover, the operators $L^+$ and $L^-$ satisfies the following coercivity property for the mass super-critical case.
\begin{lem}[Coercivity]
\label{spectalprop}
There exists $C >0$ such that for all $h=h_1+i h_2 \in H^1(\R^3),$ we have
\begin{equation}
\label{estimation-v-1}
\left\|h\right\|_{H^1}^2 \leq  C \bigg[\left(L^+h_1,h_1\right)+ \left(L^-h_2,h_2\right)+ \sum_{j=1}^3\left(\int \partial_{x_j} Q \,   h_1 \right)^2+  \left(\int Q  h_2\right)^2 
\end{equation}
$ \hspace*{3.3cm} + \displaystyle  \left(\im \int \mathcal{Y}^{+ }\bar h\right)^2 +  \left( \im \int \mathcal{Y}^{-} \bar h \right)^2 \bigg]. $
\end{lem}
\begin{proof}
The proof of this result is well known and for the sake of completeness, we will give it in Appendix \ref{Appendix du theorem pranp}. 
\end{proof}
\begin{rem}
The scalar product $\left(L^+h_1,h_1\right)$ and $ \left(L^-h_2,h_2\right)$ must be
understood in the sense of the quadratic form $\int \left| \nabla h_1 \right|^2 +  \left| \nabla h_2 \right|^2 + \left| h \right|^2 -\int pQ^{p-1}h_1^2 - \int Q^{p-1}h_2^2 .      $ 
\item Moreover, Lemma \ref{spectalprop} is still valid with $h \in H^1_0(\Omega)$ . Indeed, $h$ can be extended to a $H^1(\R^3)$ function by letting $h(x)=0$ for $x \in \Theta.$
\end{rem}

Finally, we extend the Proposition \ref{eigen-val-funct-L} to the linearized operator $\mathcal{L}_\omega$ around the stationary soliton $e^{i t \, \omega }Q_\omega$, by a simple scaling argument. 
\begin{corll}[\cite{MR3124722}]
Let $\omega>0$ and $h \in \C$ such that $h=h_1+h_2$. The linearized operator 
$\mathcal{L}_{\omega}$ is defined by $$\mathcal{L}_{\omega} h =-L^{-}_{\omega}  \, h_2 + i \,  L^{+}_{\omega} \, h_1, $$ where, \begin{equation*}
    L^{+}_{\omega} \, h_1= - \Delta h_1 + \omega h_1 - p Q_\omega^{p-1} h_1 \qquad   \text{and}   \qquad  L^{-}_{\omega}  \, h_2 = -\Delta h_2 + \omega h_2 - Q_\omega^{p-1} h_2. 
    \end{equation*}
Moreover, the spectrum $\sigma (\mathcal{L}_\omega)$ of $\mathcal{L}$ satisfies $$ \sigma(\mathcal{L}_\omega ) \cap \R = \{-e_\omega,0,e_\omega\},\; \text{ where } \; e_\omega= \omega^{\frac{3}{2}} e_0>0. $$ 
Furthermore, $e_\omega$  and $-e_\omega$ are simple eigenvalues of  $\mathcal{L}_\omega$ with eigenfunctions $\mathcal{Y}^{+}_\omega$ and $\mathcal{Y}^{-}_\omega$ 
$$   \mathcal{L}_\omega \mathcal{Y^{\pm}_{\omega}} = \pm e_\omega \mathcal{Y^{\pm}_{\omega}} , $$
where, 
$$ \mathcal{Y^{\pm}_{\omega}}(x)= \omega^{\frac{1}{4} } \mathcal{Y^{\pm }}( \sqrt{\omega} x)  \quad \text{and} \quad \mathcal{Y^{+}_{\omega}}= \overline{\mathcal{Y^{-}_{\omega}}} . $$

\end{corll}
\begin{rem}
The null-space of $L_{\omega}^{+}$ is spanned by $\partial_{x_1}Q_\omega$, $\partial_{x_2}Q_\omega$ and $\partial_{x_3}Q_\omega$ and the null-space of $L^{-}_{\omega}$ is spanned by $Q_\omega$.
\end{rem}

\subsection{Compactness argument } 
\label{compactness argumenttt}
Denote: 
\begin{align*}    
R(t,x)& = Q_{\omega}(x-t \, v)\Psi(x) e^{i\varphi(t,x)}\\ 
Y_{\pm}(t,x) &=\mathcal{Y}^{\pm}_{\omega}(x-tv) \Psi(x) e^{i\varphi(t,x)} ,  
\end{align*} 
where,  $\varphi(t,x) = \frac{1}{2}(x.v)- \frac{1}{4} |v|^2 t + t \, \omega . $ \\ 
Let $T_n \rightarrow \infty$ be an increasing sequence of times. 
\begin{prop}
\label{mainprop}
There exists $n_0 \geq 0$, $T_0 >0$ and $C>0$ (independent of $n$) such that the following holds. For each $ n \geq n_0$ there exists $\lambda_n := (\lambda_n^{\pm})_n \in \R^{2}$ such that
$$ \left|\lambda_n \right| \leq e^{-\delta \sqrt{\omega} |v| T_n } \, ,$$

and the solution $u_n$ of
\begin{equation}
\label{u_n}
\begin{cases}
i \partial_t u _n + \Delta u_n = - |u_n|^{p-1}u_n  ,\\
u_n(T_n) = R(T_n) + i  \, \lambda_n^{\pm} Y_{\pm}(T_n) , 
\end{cases} 
\end{equation}
 is defined on the interval time $[T_0 , T_n]$ and satisfies
 \begin{equation}
 \label{unifor_estimate}
  \forall t \in [T_0,T_n] \qquad  \|u_n(t) - R(t) \|_{H^1_0(\Omega)} \leq C e^{-\delta \sqrt{\omega} |v| t }.
\end{equation}
\end{prop}
\begin{proof}
We will assume this proposition to prove Theorem $\ref{theorem-pranp}$ and we postpone the proof of Proposition $\ref{mainprop}$ to Section $3$.  
\end{proof}
 Now, we will start the proof of the Theorem \ref{theorem-pranp} 
 assuming the main Proposition \ref{mainprop}. The proof is based on the compactness argument and the uniform estimate \eqref{unifor_estimate}. \\ 
 Renumbering the indices, we can take $n_0=0$ in Proposition \ref{mainprop}.
\begin{proof}[Proof of Theorem $\ref{theorem-pranp}$ assuming Proposition $\ref{mainprop}$] 
The proof proceeds in several steps.
\begin{itemize}
\item Step 1 :  " compactness argument " 
The Proposition $\ref{mainprop}$ implies that there exists a sequence $u_n(t)$ of solution defined on $[T_0,T_n]$ such that  $$ \forall n \in \N, \; \forall t \in [T_0,T_n], \quad   \|u_n(t) - R(t) \|_{H^1_0(\Omega)} \leq C e^{-\delta \sqrt{\omega}|v| t } . $$

\begin{lem}
\label{limsup}
\begin{equation*}
\displaystyle \lim_{M\rightarrow +\infty} \sup_{\,n\in \mathbb{N}}  
\int_{|x| \geq M}  u_n^2(T_0,x)  \, dx =0 .
\end{equation*}
\end{lem}
\begin{proof}
The proof of the lemma is the same as in \cite{MR2271697} for the construction of multi-soliton solutions of (NLS) for the subcritical case on $\R^{d}.$ We give it for the sake of complet. \\  

Let $\varepsilon > 0 $ and $T_{\varepsilon} \geq T_0 $ such that: $C^{^2} e^{-2 \delta \sqrt \omega \left|v\right| T_{\varepsilon} }< \varepsilon$, where  $C$ and $ \delta $  are the same constant as in the Proposition \ref{mainprop}.\\
For $n$ large enough, so that $T_n \geq T_{\varepsilon}$ and due to \eqref{unifor_estimate}, we have \begin{equation*}
    \int_{\Omega} \left|u_n(T_{\varepsilon}) - R(T_{\varepsilon})  \right|^{2} dx \leq C^{^2}e^{-2 \delta \sqrt \omega \left|v\right| T_{\varepsilon} } \leq \varepsilon .
\end{equation*}
Let $M(\varepsilon) > 0$ such that $$ \int_{|x| \geq M(\varepsilon)} \left| R(T_{\varepsilon})  \right|^{2} \, dx < \varepsilon ,$$ 
by direct computation,
\begin{equation*}
\int_{\left|x\right| \geq M(\varepsilon)} \left| u_n(T_{\varepsilon}) \right|^2  dx \leq 4 \varepsilon .
\end{equation*}

Now consider a $C^1$ cut-off function $f:\R \longrightarrow [0,1]$ such that \\
$f\equiv 0 $ on  $ ]-\infty,1]$ $\qquad;\qquad $ $0<f'<2$  on $(1,2)$ $\qquad ; \qquad $ $f \equiv 1$ on $(2,+\infty).$

For $K_{\varepsilon}>0$ to be specified later, we can check that
\begin{equation}
\label{eq_u_n-f}
\frac{d}{dt} \int_{\Omega} \left|u_n(t)  \right|^2 f\left(\frac{\left|x \right|- M(\varepsilon)}{K_{\varepsilon} } \right)dx = \frac{-2}{K_{\varepsilon}} \im \int_{\Omega} u_n(t) \left(\nabla \overline{u_n}. \frac{x}{\left|x\right|}\right)  f'\left(\frac{\left|x\right|-M(\varepsilon)}{K_{\varepsilon}} \right)dx.
\end{equation}
From Proposition $\ref{mainprop}$, $\exists \,  \alpha > 0,$  $\forall n$ and $\forall t \geq T_0, \;  \left\|u_n(t) \right\|_{H^1_0}^2  \leq \alpha.$ Using $\eqref{eq_u_n-f}$ we get 
\begin{align*}
\left|\frac{d}{dt}  \int_{\Omega} \left| u_n(t) \right|^2  f\left(\frac{\left|x\right|- M(\varepsilon)}{K_{\varepsilon}}\right) \right|  \leq \frac{4}{K_{\varepsilon} }\left\| u_n(t) \right\|_{H^1_0}^2 \leq \frac{ 4 }{K_{\varepsilon}}  \, \alpha .
\end{align*}
Now, we choose $K_{\varepsilon} > 0$ independently of $n$ such that $$ K_{\varepsilon}  \geq  \left(\frac{T_\varepsilon- T_0}{\varepsilon}  \right)  \,4 \, \alpha , $$
which yields$$\left|\frac{d}{dt} \int_{\Omega} \left|u_n(t) \right|^2 f\left(\frac{\left|x\right|-M(\varepsilon)}{K_\varepsilon} \right) \right| \leq \frac{\varepsilon}{T_\varepsilon-T_0} \,. $$
Integrating on the time interval $\left[T_0,T_\varepsilon\right]$, we get
\begin{align*}
& \int_{\Omega} \left|u_n(T_0) \right|^2 f\left(\frac{\left|x\right|-M(\varepsilon)}{K_\varepsilon}\right) dx  - \int_{\Omega} \left| u_n(T_\varepsilon) \right|^2 f\left(\frac{\left|x\right|-M(\varepsilon)}{K_\varepsilon}\right)  dx \\ &\leq \int_{T_0}^{T_\varepsilon} \left| \frac{d}{dt} \int \left| u_n(t) \right|^2 f\left(\frac{\left|x\right|-M(\varepsilon)}{K_\varepsilon}\right)  dx \right| dt  \\ & \leq \varepsilon .
\end{align*}
Hence, $$ \int_{\Omega} \left| u_n(T_0) \right|^2 f\left(\frac{\left|x\right|-M(\varepsilon)}{K_\varepsilon}\right) dx \leq \varepsilon + \int_{\Omega} \left|u_n(T_\varepsilon) \right|^2  f\left(\frac{\left|x\right|-M(\varepsilon)}{K_\varepsilon}\right) dx . $$
Due to the properties of $f$, we have\begin{align*}
\int_{\left|x\right|>2 K_{\varepsilon}+ M(\varepsilon) } \left|u_n(T_0)\right|^2 dx &\leq \int_{\Omega} \left|u_n(T_0)\right|^2 f\left(\frac{\left|x\right|-M(\varepsilon)}{K_\varepsilon}\right) dx  \\  & \leq  \varepsilon +\int_{\Omega} \left|u_n(T_\varepsilon) \right|^2  f\left(\frac{\left|x\right|-M(\varepsilon)}{K_\varepsilon}\right) dx \\ & \leq \varepsilon + \int_{\left|x\right| \geq M(\varepsilon)} \left| u_n(T_\varepsilon) \right|^2 dx \\ & \leq \varepsilon + 4 \varepsilon= 5 \varepsilon . 
\end{align*}
This concludes the proof of the lemma.
\end{proof}
Due to the main proposition we have$$\left\|u_n(T_0)\right\|_{H^1_0(\Omega)} \leq \alpha .$$
Since $H^1_0$ is a Hilbert space, there exists a subsequence of $(u_n(t))_n$ that we still denote by $(u_n(t))_n$ to simplifies notation and $\mathcal{U}_0 \in H^1_0(\Omega) $ such that $$ u_{n}(T_0) \rightharpoonup \mathcal{U}_0  \quad in \quad H^1_0(\Omega),  \;  \text{as} \; \; n \longrightarrow +\infty .$$
By the compactness of the embedding of $H^1(\{ \left|x \right| \leq A \})$ into $L^2(\{ \left|x \right| \leq A \})$, we have $$u_n(T_0) \longrightarrow \mathcal{U}_0 \quad \text{ in} \quad L^2_{loc} \, .$$
By the Lemma $\ref{limsup}$, we get $u_n(T_0) \longrightarrow \mathcal{U}_0 \quad \text{ in} \quad L^2(\Omega).$ \\ 

Now using the following interpolation inequality
 $$\forall s\in[0,1[, \quad  \left\|u_n(t) - \mathcal{U}_0 \right\|_{H^s(\Omega)} \leq \left\|u_n(t) - \mathcal{U}_0 \right\|_{L^2(\Omega)}^{1-s} \; \left\|u_n(t) - \mathcal{U}_0 \right\|_{H^1_0(\Omega)}^{s}  , $$ we obtain,\begin{equation}
\label{convergenceHs}
u_n(T_0) \longrightarrow \mathcal{U}_0 \quad in \quad H^s(\Omega) ,  \; \forall s \in [0,1). 
\end{equation}
\end{itemize}
\begin{itemize}

\item Step 2: Construction of the solution. 
\begin{lem}[Well posedness in $H^s(\Omega),s\in[s_p,1)$ ] 
\label{wellposed}
Denote by $s_p= \frac{3}{2} - \frac{3}{p+1}$ and let $s \in [0,1)$ such that $s_p \leq s$.\\
Let $u_0 \in H^s(\Omega)$ then there exists $\tau>0$ which depends only on $\left\|u_0 \right\|_{H^s(\Omega)}$, such that for any $T \in [0,\tau]$, the nonlinear Schr\"odinger equation (NLS$_\Omega$) admits a unique solution $u \in C([0,T],H^s(\Omega))$.\\

Furthermore, the solution $u$ can be extended to a maximal existence interval $ [0,T_+)$ and the following alternative holds,\\
Either $T_+= +\infty$ (the solution is global) or $T_+ < +\infty$ (the solution blows up in finite time) and $$\lim\limits_{t \to T_+ } \left\|u(t, \cdot) \right\|_{H^s}=+\infty . $$ 
\end{lem}
\begin{proof}{see Appendix \ref{Appendix du theorem pranp}.}
\end{proof}
Due to the Lemma $\ref{wellposed}$, the equation  (NLS$_{\Omega})$ is well-posed in $H^s(\Omega)$, for $ s_p \leq s <1$.\\ 
Let $\widetilde{u}$ be the maximal solution of\begin{equation}
\begin{cases}

	i\partial_t \widetilde{u}+\Delta \widetilde{u}= -|\widetilde{u}|^{p-1}\widetilde{u} \qquad  \forall (t,x)\in [T_0,\widetilde{T})\times\Omega, \\  \widetilde{u}(T_0,x) =\mathcal{U}_0  \qquad\qquad \qquad\forall x \in \Omega , \\
	\widetilde{u}(t,x)=0 \qquad\qquad \qquad  \quad \forall(t,x)\in[T_0,\widetilde{T})\times \partial\Omega . 
	
\end{cases}
\end{equation}
By \eqref{convergenceHs} we have
\begin{equation}
\label{convergenceHsp}
u_n(T_0) \longrightarrow \mathcal{U}_0 = \widetilde{u}(T_0,x)  \quad in \quad H^s(\Omega) , \,\forall s \in [s_p,1). 
\end{equation}

For $n$ large enough, $ u_n(t)$ is defined for all $ t \in [T_0,\widetilde{T})$ and by the continuity of the flow  we have$$u_n(t) \longrightarrow \widetilde{u}(t)  \quad in \quad H^s(\Omega), \quad   \forall s \in [s_p,1). $$

Due to the main Proposition $\ref{mainprop}$, we know that for $n$ large enough $u_n(t)$ is uniformly bounded in $H^1_0$. Then necessarily, $$ \forall t \in [T_0,\widetilde{T}), \quad u_n(t) \rightharpoonup \widetilde{u}(t) \quad \text{ in } \; \,  H^1_0(\Omega).  $$
Using the property of weak convergence and by the main proposition, it follows that $$\forall t \in [T_0,\widetilde{T}), \quad  \left\|\widetilde{u}(t)- R(t) \right\|_{H^1_0} \leq \lim \inf \left\|u_n(t)-R(t)\right\|_{H^1_0} \leq C e^{-\delta \sqrt{\omega} |v| t } .$$

In particular we deduce that, $\widetilde{u}$ is bounded in $H^1_0(\Omega)$. Due to the blow up alternative we get $\widetilde{T}=+\infty$. Finally, we have $\widetilde{u} \in C ([T_0,+\infty), H^1_0(\Omega))$  and by \eqref{unifor_estimate} in Proposition~\ref{mainprop}, $$\forall t \in [T_0,+\infty), \quad \left\| \widetilde{u}(t)-R(t)\right\|_{H^1_0} \leq e^{-\delta \sqrt{\omega} |v| t },$$ 
which concludes the proof of the Theorem \ref{theorem-pranp}
\end{itemize}\end{proof}
\section{Proof of the uniform estimate  }
\label{Proof of the uniform estimate}
\subsection{Bootstrap and topological arguments} 
\label{Bootstrap and topological arguments}
In this section, we prove the main Proposition~$\ref{mainprop}$. We use some modulation in the phase and translation parameters in the decomposition of the solution to obtain the orthogonality conditions. Next, we use a bootstrap argument to control these parameters and some scalar product that are related to the size of the soliton. Finally, to conclude the proof we use a topological argument to control the unstable direction.    

\begin{rem}
In this section, to simplify notations we will write $r$ instead of $r_\omega$ and we will drop the index $n$ for most variables. Hence, we will write $u$ for $u_n$, $\lambda^{\pm}$ for $\lambda^{\pm}_n$ etc.  Except the sequence of times that will be written with the index. As Proposition $\ref{mainprop}$ is proved for given $n$, this should not be a source of confusion. We possibly drop the first terms of the sequence $T_n$, so that, for all $n$, $T_n$ is large enough for our purposes. 
\end{rem}
\subsubsection{Modulated final data}
    \begin{lem}[modulation for time independent function]
    \label{modulation}
    There exists $C,\epsilon>0 $ such that the following holds.\\ Given $\alpha \in \R^3 $ and $\theta \in \R.$
    If $u(x) \in L^2$ is such that $$ \left\| u - R \right\|_{L^2} \leq \epsilon .$$ 
    Then there exists modulation parameters $y=(y_i)_i \in \R^3$ and $\mu \in \R$,
    such that setting$$r(x)= u(x) - \widetilde{R}(x),$$
    the following holds $$ \left\|r\right\|_{L^2}+ \left|y\right|+|\mu| \leq C \left\|u-R \right\|_{L^2},$$
    and
    $$\re \int r(x)  \partial_{x_j}{\widetilde{Q}}_\omega(x) \Psi(x) e^{-i\, (\frac{1}{2}(x.v) + \theta)} e^{-i \mu} dx= \im \int r(x) \overline{\widetilde{R}}(x) dx  = 0, \;  j=1,2,3,  $$
    
    where,
\begin{align*}
\hspace*{-3.5cm} R(x)&=Q_\omega(x-\alpha)\Psi(x)e^{i\, (\frac{1}{2}(x.v) + \theta)} ,  \\ \hspace*{-3.5cm}\widetilde{Q}_\omega(x)&=Q_\omega(x-\alpha- y), \\ 
\hspace*{-3.5cm}\widetilde{R}(x)&=\widetilde{Q}_\omega(x) \Psi(x) e^{i\, (\frac{1}{2}(x.v) + \theta)} e^{i \mu}  .
\end{align*}
    Furthermore, $u \longmapsto (r,y,\mu)$ is a smooth $C^1$-diffeomorphism. 
    \end{lem}
 \begin{proof}{see Appendix \ref{Appendix du theorem pranp}.}
  
\end{proof}
Note that the previous lemma applies to time independent functions. A consequence of this modulation in the decomposition of fixed $u$ is the the following result on a solution $u(t)$ of~$\eqref{u_n}$.
 \begin{corll}
 \label{modulated-u(t)}
 There exists $C,\epsilon >0$ such that the following holds for all $t \in [T,T_n]$, for $T>T_0$, if $u(t,\cdot) \in L^2_x$ satisfies \begin{equation*}
 \label{mod-u(t)-R(t)-L^2 }
 \left\| u(t) - R(t) \right\|_{L^2} \leq \epsilon  . 
  \end{equation*} 
    Then there exits a $C^1$-functions $y:[T,T_n] \longrightarrow \R^3 $ and $\mu : [T,T_n] \longrightarrow \R$ such that if we set  $$r(t,x) = u(t,x)-\widetilde{R}(t,x)  , $$
the following holds $$ \left\|r(t)\right\|_{L^2}+ \left|y(t)\right|+|\mu(t)| \leq C \left\|u(t)-R(t) \right\|_{L^2},$$
    and
    \begin{align}
\label{ortho-conti-for-time-dep-Real}
  \re &\int r(t,x)  \partial_{x_j}{\widetilde{Q}}_\omega(t,x) \Psi(x) e^{-i\, (\frac{1}{2}(x.v) + \theta(t))} e^{-i \mu(t)} dx= 0 \; \,   j=1,2,3,   \\ 
  \label{ortho-conti-for-time-dep-im}
  \im& \int r(t,x) \overline{\widetilde{R}}(t,x) dx  = 0,
    \end{align}
      
    where,\begin{align*}
            &R(t,x)=Q_\omega(x-\alpha(t)) \, \Psi(x)e^{i\, (\frac{1}{2}(x.v) + \theta(t))}\; , \text{ with } \alpha(t):= t \, v \text{ and } \; \theta(t):= - \frac{1}{4} |v|^2 t + t \,  \omega \, .  \\
        &\widetilde{Q}_\omega(t,x)=Q_\omega(x-\alpha(t)- y(t))\, .\\ 
    &\widetilde{R}(t,x)=\widetilde{Q}_\omega(t,x) \Psi(x) e^{i\, (\frac{1}{2}(x.v)+ \theta(t) )} e^{i  \mu(t)} \, .   
    \end{align*} 
 \end{corll}
 
 \begin{proof}
 For small $\lambda$, the solution $u(t)$ is closed to the soliton $R(t)$ for $t$ close to $T_n$. Assume that $u(t)$ satisfies \eqref{mod-u(t)-R(t)-L^2 } on $[T,T_n]$. Applying Lemma \ref{modulation} to u(t) for any   $ t \in [T,T_n]$ and since the map $t \longmapsto u(t) $ is continuous in $H^1_0$, we obtain the existence of continuous functions 
$y:[T,T_n] \longrightarrow \R^3 $ and $\mu : [T,T_n] \longrightarrow \R$ such that \eqref{ortho-conti-for-time-dep-Real} and \eqref{ortho-conti-for-time-dep-im} holds. 
\end{proof}

 \vspace{0.5cm}
\textbf{Notation:}   $u(t)$ is defined and modulable around $R(t)$ for $t$ close to $T_n$, in the sense of the previous Corollary.  \begin{align*}
 R(t,x) &=Q_{\omega}(x-tv) \Psi(x)  e^{i \varphi(t,x)} \; ,    \quad \;   \text{where } \; \; \varphi(t,x)=\frac{1}{2}x.v- \frac{1}{4} |v|^2 t + t \,  \omega\, .  \\
 \widetilde{Q}_{\omega}(t,x) &= Q_{\omega}(x-tv-y(t)) \, .   \\ 
 \widetilde{R}(t,x) &=\widetilde{Q}_{\omega}(t,x) \Psi(x)  e^{i\widetilde{\varphi}(t,x)}\; , \quad    \;  \; \, \quad   \text{ where } \; \;  \widetilde{\varphi}(t,x) =  \frac{1}{2}x.v- \frac{1}{4} |v|^2 t + t\,  \omega + \mu(t).    \\  
\mathcal{ \widetilde{Y}}_{\omega}^{\mp}(t,x) &=\mathcal{Y}^{\mp}_{\omega}(x-tv-y(t)).  \\
\widetilde{Y}_{\mp}(t,x)&=\mathcal{ \widetilde{Y}}_{\omega}^{\mp}(t,x) \Psi(x) e^{i\widetilde{\varphi}(t,x)} \quad \text{ and } \quad 
 \displaystyle \alpha^{\pm}(t)= \im \int \widetilde{Y}_{\mp}(t,x) \overline{r}(t,x) dx . \\
 \widetilde{L}_\omega^+ h_1 &= -\Delta h_1 + \omega h_1 - p \widetilde{Q}_{\omega}^{p-1}h_1 \quad \text{ and } \quad \widetilde{L}_\omega^-h_2 = - \Delta h_2 +\omega h_2 -\widetilde{Q}_{\omega}^{p-1}h_2.
\end{align*}

\begin{lem}[Modulated final data]
\label{modulatedfinaldata}
There exists $C>0$ (independent of $n$) such that \\ for all $\alpha^{+} \in B_{\R}(e^{-\delta \sqrt{\omega} |v| T_n})$ there exists a unique $\lambda$ such that $$\left| \lambda \right| \leq C \left| \alpha^{+} \right|,$$
and the modulation parameters $(r(T_n),y(T_n),\mu(T_n))$ of $u(T_n)$ satisfies
\begin{equation}
\begin{aligned}
\begin{cases}
\alpha^{+}(T_n)&= \alpha^{+}, \\
\alpha^{-}(T_n)&=0.
\end{cases}
\end{aligned}
\end{equation}

\end{lem}

\begin{proof}[Proof. See Appendix \ref{Appendix du theorem pranp}]

\end{proof}
Let $T_0$ to be specified later, independent of $n$. Let $\alpha^{+}$ to be chosen, $\lambda$ be given by Lemma~\ref{modulatedfinaldata} and let u be the corresponding solution of $\eqref{u_n}$. We now define the maximal time interval $[T(\alpha^{+}),T_n]$, on which suitable exponential estimates hold. 
\begin{defi}
Let $ T(\alpha^{+})$ be the infimum of $T \geq T_0$ such that the following properties hold for all $t \in [T,T_n] :$ \\
Closeness to $R(t):$ $$\left\|u(t)-R(t)\right\|_{H^1_0} \leq \varepsilon \, . $$

In particular, this ensures that $u(t)$ is modulable around $R(t)$ in the sense of Lemma \ref{modulation}.\\
Estimates on the modulation parameters: There exists $M >0$ and $M'>0$ to be specified later, 
\begin{align}
\label{p-estimate}
\left\|r(t)\right\|_{H^1_0} &\leq M e^{-\delta \sqrt{\omega} |v| t}  \\
\label{y-estimate}
\left| y(t) \right| &  \leq M' e^{-\delta \sqrt{\omega} |v| t}    \\  
\label{mu-estimate}
\left| \mu(t) \right| &\leq  M'  e^{-\delta \sqrt{\omega} |v| t}\\
\label{a±-estimate}
\left|\alpha^{\pm}(t)\right| &\leq e^{-\delta \sqrt{\omega} |v| t} .
\end{align}

\end{defi}

Note that, if for all $n$ we can find $\alpha^{+}$ such that $T(\alpha^{+})=T_0$ then the  Proposition $\ref{mainprop} $ is proved. It remains to prove the existence of such value of $ \alpha^{+}$. \\

Denote $h(t,x)= e^{-i \widetilde{\varphi}(t,x)}r(t,x). $ Recall that, \begin{align*} 
u(t,x)&= \widetilde{R}(t,x)+r(t,x)\\
      &=e^{i\widetilde{\varphi}(t,x)} (\widetilde{Q}_{\omega}(t,x)\Psi(x)+h(t,x)) .
\end{align*}
\begin{lem}
\label{equation-de-h+dy+da+}
Let $ t \in[T(\alpha^{+}),T_n]$ and let $C,\delta >0 $. We have 
\begin{multline}
\label{equation-on-h}
i\partial_t h+\Delta h-\omega h +(\frac{p+1}{2}) {\widetilde{Q}_{\omega}}^{p-1}\Psi^{p-1}h+(\frac{p-1}{2}){\widetilde{Q}_{\omega}}^{p-1} \Psi^{p-1}\overline{h}+i\, v.\nabla h-\frac{d\mu(t)}{dt}h \\ +{\widetilde{Q}_{\omega}}^p\Psi(\Psi^{p-1}-1)
+2\nabla \widetilde{Q}_{\omega} \nabla \Psi + \widetilde{Q}_{\omega} \Delta \Psi+i \, v \, \widetilde{Q}_{\omega} \nabla \Psi - i \, \frac{dy(t)}{dt} \nabla \widetilde{Q}_{\omega} \Psi  - \frac{d\mu(t)}{dt} \widetilde{Q}_{\omega} \Psi + 
\beta(t,x)=0,
\end{multline}
where $\beta(t,x)$ is a remainder terms on $h$. 
\begin{equation}
\label{estimation-en-mu-et-y}
\left|\frac{d\mu(t)}{dt}\right|+\left|\frac{dy(t)}{dt}\right| \leq C \left\| h(t) \right\|_{H^1_0}^2 + Ce^{-2\delta \sqrt{\omega} |v| t }. 
\end{equation}
\begin{equation}
\label{estimation-dt-alpha}
\left|\frac{d\alpha^{\pm}(t)}{dt} \pm e_\omega \alpha^{\pm}(t) \right| \leq C \left\|h (t) \right\|_{H^1_0}^3 + C e^{-2\delta \sqrt{\omega} |v| t }.
\end{equation}
\end{lem}
\begin{proof}
For the equation \eqref{equation-on-h} of $h$ it suffices to plug the above expression of $u(t,x)$ on the nonlinear Schr\"odinger equation : $i\partial_t u + \Delta u = -|u|^{p-1}u.$  Using, the elliptic equation $\eqref{eq_Q}$ of $Q_\omega $ and the Taylor expansion for the nonlinear term, we get \eqref{equation-on-h}, with $ \left\|\beta(t)\right\|_{L^2} \leq  C \left\|h(t)\right\|^2_{H^1_0}.$  \\ 
For the proof of \eqref{estimation-en-mu-et-y} and \eqref{estimation-dt-alpha}, we claim the following estimates.
\begin{clm}
\label{derive-Condition-orthogonal}
\begin{equation*}
 \im \int \partial_t \overline{h}(t,x) \widetilde{Q}_{\omega}(t,x) \Psi(x) dx  = \displaystyle \sum_{k=1}^{3} \im \int \overline{h}(t,x)  (v_k+\frac{d y_{k}}{dt}(t)) \;  \partial_{x_k} \widetilde{Q}_{\omega}(t,x) \Psi(x) dx .
 \end{equation*}
 
 \begin{equation*}
\re \int  \partial_t \overline{ h}(t,x)  \partial_{x_j}\widetilde{Q}_{\omega}(t,x)  \, \Psi(x) dx = \sum_{k=1}^{3} \re \int \overline{h}(t,x) (v_k+\frac{d y_k }{dt}(t)) \partial_{x_k} \partial_{x_j} \widetilde{Q}_{\omega} (t,x) \Psi(x) dx \, , j=1,2,3 .
\end{equation*}
\end{clm}
\begin{proof}It is just a consequence of the orthogonality conditions in Lemma \ref{modulation}. So, we have \begin{equation*}
   \re \int h(t,x) \partial_{x_j} \widetilde{Q}_{\omega}(t,x) \,  \Psi(x) dx= \im \int h(t,x) \widetilde{Q}_{\omega}(t,x) \Psi(x) dx  = 0 , \;  j=1,2,3.
\end{equation*}
Differentiating each equality with respect to the time variable $t$, the Claim \ref{derive-Condition-orthogonal} follows. 
\end{proof}

Now let us estimate  $\frac{dy}{dt}(t)$ and $\frac{d \mu(t)}{dt}$ in \eqref{estimation-en-mu-et-y}. Multiply by $\partial_{x_j} \widetilde{Q}_{\omega}\Psi$ and take the imaginary part of the equation $\eqref{equation-on-h}$. Using the Claim \ref{derive-Condition-orthogonal} and the fact that $Q_\omega$ is radial, so that
 \begin{equation} \begin{cases}
\quad  \; Q_{\omega}(x_1,x_2,x_3)= Q_{\omega}(-x_1,x_2,x_3) ,\\
\partial_{x_1}Q_{\omega}(x_1,x_2,x_3) = - \partial_{x_1} Q_{\omega}(-x_1,x_2,x_3).
\end{cases}
\end{equation}
which yields
\begin{equation*}
\int \partial_{x_1}Q_{\omega}(x_1,x_2,x_3)   Q_{\omega}(x_1,x_2,x_3)\,  dx = - \; \int \partial_{x_1} Q_{\omega}(x_1,x_2,x_3) Q_{\omega}(x_1,x_2,x_3)\,  dx .
\end{equation*}
Hence $$\int \partial_{x_j} Q_{\omega}(t,x) \; Q_{\omega}(t,x) \, dx = 0 , \quad \text{ for } j=1,2,3.$$
We obtain the following equality on $\frac{dy(t)}{dt}$.
\begin{align*}
 \frac{dy_j(t)}{dt} \| \partial_{x_j} \widetilde{Q}_{\omega} \Psi \|_{L^2}^2 &= \underbrace{ \int h_1(t,x) \,  \frac{dy(t)}{dt}. \nabla ( \partial_{x_j} \widetilde{Q}_{\omega}(t,x) ) \Psi(x) dx}_{\rm I^{y}_h} - \underbrace{\frac{d \mu(t)}{ dt } \int h_2(t,x) \partial_{x_j} \widetilde{Q}_{\omega}(t,x) \Psi(x) dx }_{\rm I^{\mu}_h } \\ &- \underbrace{ \int  \widetilde{L}^{-}_{\omega} h_2(t,x)  \partial_{ x_j} \widetilde{Q}_{\omega}(t,x) \Psi(x) dx + \int h_2(t,x)  \widetilde{Q}_{\omega}^{p-1}(t,x)  ( \Psi^{p-1}(x) - 1 ) dx}_{\rm I_h^1} \\ & \underbrace{+ \int h_1(t,x) \, \partial_{x_j} \widetilde{Q}_{\omega}(t,x)  \, v. \nabla \Psi(x) dx}_{ \rm I_h^2}+ O( \left\| h(t) \right\|^2_{H^1_0} ). 
\end{align*}

Taking the scalar product with $\widetilde{Q}_{\omega}(x) \Psi$ and the equation $\eqref{equation-on-h}$ on $h$. Using the same argument as above, we get the following equality on $\frac{d\mu(t)}{dt}$. 
\begin{align*}
\frac{d\mu(t)}{dt} \| \widetilde{Q}_{\omega} \Psi \|_{L^2}^2 & =
\underbrace{\int h_2(t,x) \frac{dy(t)}{dt}. \nabla \widetilde{Q}_{\omega}(t,x) \Psi(x)  dx     }_{ \rm J_h^y} - \underbrace{\int \frac{d \mu(t)}{dt}  h_1(t,x) \widetilde{Q}_{\omega}(t,x) \Psi(x) dx }_{\rm J_h^\mu} \\ &- \underbrace{\int \widetilde{L}^{+}_{\omega} h_1(t,x) \widetilde{Q}_{\omega}(t,x) \Psi(x)dx + \int p \widetilde{Q}_{\omega}^{p-1}(t,x) h_1(t,x) (\Psi^{p-1}(x) - 1) dx}_{\rm J_h^1} \\ &- \underbrace{\int h_2(t,x) \widetilde{Q}_{\omega}(t,x) v.\nabla \Psi(x) dx}_{\rm J_h^2} + \underbrace{\int \widetilde{Q}_{\omega}^{p+1}(t,x) \Psi^2(x)(\Psi^{p-1}(x) - 1 )}_{\rm J_1} \\ & + \underbrace{\int {\widetilde{Q}_{\omega}}^2(t,x) \Delta \Psi(x) \Psi(x) dx}_{\rm J_2} + O\left( \left\| h(t) \right\|^2_{H^1_0} \right).
\end{align*}

Summing the absolute values of the two equalities above and using the fact that  $$\| \widetilde{Q}_{\omega} \Psi \|_{L^2}^2= \left\|Q_\omega \right\|_{L^2}^2 + O(e^{-2\delta \sqrt{\omega} \left| v \right| t } ) \quad \text{ and  }  \quad \| \nabla \widetilde{Q}_{\omega} \Psi \|^2_{L^2} = \left\| \nabla Q_\omega   \right\|^2_{L^2} + O(e^{-2\delta \sqrt{\omega} \left|v\right| t } ) , $$ 
We obtain the left hand side on the estimate \eqref{estimation-en-mu-et-y}
Next, we have to estimate the right hand side in both equalities. 
\begin{align*}
    \left| \rm I^y_h \right|:= \left| \int h_1(t,x) \frac{dy(t)}{dt} . \nabla (\partial_{x_j} \widetilde{Q}_{\omega}(x) ) \Psi(x) dx \right| &\leq C  \left| \frac{dy(t)}{dt}  \right|  \left\| h(t) \right\|_{L^2}   \\ &\leq C_1  \left| \frac{dy(t)}{dt} \right| \, M e^{-\delta \sqrt{\omega} \left| v \right| T_0 }  \\ & \leq  \frac{1}{10} \left| \frac{dy(t)}{dt} \right| 
    \left\| \partial_{x_j} Q_\omega   \right\|_{L^2}^2 .
    \end{align*}
    Provided\begin{equation}
    \label{1-condi-M}
     M e^{ - \delta \sqrt{\omega} \left|v \right| T_0} \leq \frac{1}{10 \, C_1} \left\| \partial_{x_j} Q_\omega   \right\|_{L^2}^2, \;  j=1,2,3. 
\end{equation}
\begin{align*}
\left| \rm I^\mu_h \right|:= \left| \frac{d\mu(t)}{dt} \int h_2(t,x) \partial_{x_j} \widetilde{Q}_{\omega}(x) \Psi(x) dx \right|
& \leq C \left| \frac{d \mu(t)}{ dt } \right|   \left\| h(t)\right\|_{L^2}   \\ & \leq C_2  \left|  \frac{d \mu(t)}{ dt } \right|  M e^{-\delta \sqrt{\omega} \left|v \right| T_0} \\
 & \leq \frac{1}{10 } \left|  \frac{d \mu(t)}{ dt } \right| \left\|  Q_\omega  \right\|_{L^2}^2,
\end{align*}
 If the following condition is satisfied, \begin{equation}
     \label{2-condi-M}
     M e^{-\delta \sqrt{\omega} \left| v \right| T_0} \leq \frac{1}{10 C_2 } \left\| Q_\omega  \right\|_{L^2} ^2.
 \end{equation}
\begin{align*}
\left| \rm  J_h^y \right|:= \left| \int h_2(t,x) \frac{dy(t)}{dt}.\nabla \widetilde{Q}_{\omega}(x) \Psi(x) dx \right|
& \leq  C \left| \frac{dy(t)}{dt} \right|  \left\| h(t) \right\|_{L^2} \\  & \leq C  \left| \frac{dy(t)}{dt}  \right|  M e^{-\delta \sqrt{\omega } \left|v \right| T_0}\\ & \leq \frac{1}{10} \left| \frac{dy(t)}{dt}  \right| \left\| \partial_{x_j} Q_\omega  \right\|_{L^2}^2,
\end{align*}
If the condition \eqref{1-condi-M} holds. 
\begin{align*}
    \left| \rm J^\mu_h \right|:=  \left| \frac{d \mu(t)}{dt} \int h_1(t,x) \widetilde{Q}_{\omega}(x) \Psi(x) dx  \right| &\leq C  \left| \frac{d \mu(t)}{dt} \right| \left\| h(t)\right\|_{L^2} \\ &\leq  C \left| \frac{d \mu(t)}{dt} \right| M e^{-\delta \sqrt{\omega} \left|v \right| T_0 }\\ & \leq \frac{1}{10} \left| \frac{d \mu(t)}{dt} \right| \left\| Q_\omega \right\|_{L^2}^2,
\end{align*}
If the condition \eqref{2-condi-M} is verified.\\

We next treat the terms $\rm I_h:= \rm I^1_h +\rm I^2_h$ and $\rm J_h:=J_h^1 + J_h^2$ that depends on $h$. We will estimate the main integral for both terms, where appears the self-adjoint operator $\widetilde{L}^{+}_\omega$ and $\widetilde{L}^{-}_{\omega}\, .$ 

\begin{align*}
 \displaystyle \left|  \int \widetilde{L}^{-}_{\omega}h_2(t,x) \partial_{x_j} \widetilde{Q}_{\omega}(x) \Psi(x) dx \right|&= \left| \int h_2(t,x) \widetilde{L}^{-}_{\omega} \left( \partial_{x_j} \widetilde{Q}_{\omega}(x) \Psi(x) \right) dx  \right| \\  &\leq C \left\|  h (t)\right\|_{H^1_0}.
\end{align*}
Similarly, we can estimate the integral on $\widetilde{L}^{+}_{\omega}$. We obtain $$\left|\rm I_h \right|+ \left| \rm J_h \right|  \leq C  \left\| h \right\|_{L^2}. $$

Finally, we have to estimate $\rm J_1 $ and $\rm J_2$. Using the exponential decay of $Q$ and the fact that $\Delta \Psi$ and $(\Psi^{p-1} -1 )$ have a compact support, we get
\begin{align*}
    \left| \rm J_1 + J_2 \right|&:= \left| \int \widetilde{Q}_{\omega}^{p+1}(x) \Psi(x)^2(\Psi^{p-1} - 1 ) + \int \widetilde{Q}_{\omega}^2(x) \Delta \Psi \Psi dx \right| \\ &\leq C e^{-2\delta \sqrt{\omega } \left| v \right| t}.
\end{align*}
We have proved the estimate \eqref{estimation-en-mu-et-y},  if conditions \eqref{1-condi-M} and \eqref{2-condi-M} on $ M$ hold. For $T_0$ large enough, \begin{equation}
    \label{1+2-condi-M}
    M e^{-\delta \sqrt{\omega}  \left|  v \right| T_0  }  \leq \frac{1}{10 C^{'}} \min{ \left( \left\| \partial_{x_j} Q_\omega   \right\|^2_{L^2},  \left\| Q_\omega  \right\|^2_{L^2} \right)},
\end{equation}
where $C^{'}=\max{(C_1,C_2)}$. \\

Next, we have to prove the last estimate \eqref{estimation-dt-alpha}. Let us recall that \begin{equation*}
\alpha^{\pm}(t)=\im \int \overline{r}(t,x) \widetilde{Y}_{\mp}(t,x) dx =  \im \int \overline{h}(t,x) \widetilde{\mathcal{Y}}^{\mp}_{\omega}(t,x) \Psi(x) dx.
\end{equation*}
\begin{equation*}
\begin{split}
  \frac{d}{dt}\alpha^{\pm}(t) &= - \underbrace{ \im \int   \overline{h}(t,x) \,  \frac{dy(t)}{dt}. \nabla \widetilde{\mathcal{Y}}^{\mp}_{\omega}(t,x)  \Psi(x) dx }_{\rm I_1} - \underbrace{ \im \int   \overline{h}(t,x) \,  v. \nabla \widetilde{\mathcal{Y}}^{\mp}_{\omega}(t,x) \Psi(x) dx }_{\rm I_2}    \\& + \underbrace{\im \int \partial_t \overline{h}(t,x) \widetilde{\mathcal{Y}}^{\mp}_{\omega}(t,x) \Psi(x) dx}_{\rm I_3}
  \end{split}
\end{equation*}

 Due to \eqref{estimation-en-mu-et-y} and the  exponential decay properties of the eigenfunctions of the linearized operator. We get 
\begin{align*}
\left| \rm I_1 \right| =  \left|\im \int \overline{h}(t,x) \frac{dy(t)}{dt}. \nabla \widetilde{\mathcal{Y}}^{\mp}_{\omega}(t,x) \Psi(x) dx \right| &\leq C  \left| \frac{dy(t)}{dt} \right| \left\|h\right\|_{L^2} \\ & \leq C \left\| h(t) \right\|_{H^1_0}^3 + +C e^{-2 \delta \sqrt{\omega} |v| t } .
\end{align*}
One can check that the second integral $\rm{I_2}$ will be simplified with a term from $\rm I_3$. \\
Now, let us estimate $\rm I_3$. For this we have to use the equation \eqref{equation-on-h} of $h$. One can see that the main terms is the following  \begin{align*} 
 \partial_t \overline{h}=   - i \,  \Delta \overline{h} + i \, \omega \overline{h} - i \, (\frac{p+1}{2}) \widetilde{Q}_{\omega}^{p-1}\Psi^{p-1}\overline{h} -i \, (\frac{p-1}{2}) \widetilde{Q}_{\omega}^{p-1}\Psi^{p-1} h + f  
    \end{align*}
    Where $f$ contains all others terms of the equation \eqref{equation-on-h}. 
Let $h=h_1+ih_2$, 
\begin{align*}
- i \,  \Delta \overline{h} + i \, \omega \overline{h} - i \, (\frac{p+1}{2}) \widetilde{Q}_{\omega}^{p-1}\Psi^{p-1}\overline{h} -i \, (\frac{p-1}{2}) \widetilde{Q}_{\omega}^{p-1}\Psi^{p-1} h& = i \widetilde{L}^{+}_{\omega} h_1 + \widetilde{L}^{-}_{\omega}h_2 +  \widetilde{Q}_{\omega}^{p-1} h_2(1-\Psi^{p-1})\\ &+ i\, p \widetilde{Q}_{\omega}^{p-1}h_1(1-\Psi^{p-1}).
\end{align*}

Multiply \eqref{equation-on-h} by $\widetilde{\mathcal{Y}}^{\mp}_{\omega}(t,x) \Psi(x)$ and take the imaginary part, we obtain $\rm I_3$ on the left hand side.  
The terms containing the linearized operator will be treated later. To estimate the other terms, we use the fact that $Q_\omega$ and $\mathcal{Y}^{\mp}_\omega$ are radial, exponentially decaying at infinity and the compact support of $\nabla \Psi$ and $(1- \Psi^{p-1})$. Also, we have to use the estimate \eqref{estimation-en-mu-et-y} to obtain the right hand side of the estimate \eqref{estimation-dt-alpha}. \\ 
 
To complete the proof we have to compute the terms of the linearized operator.\\  Let $ {y}^{\mp}_1(t,x)=\re \left( \widetilde{\mathcal{Y}}_{\omega}^{\mp}(t,x) \right) $ and $ {y}^{\mp}_2(t,x)=\im \left( \widetilde{\mathcal{Y}}^{\mp}_{\omega}(t,x)\right)$. Thus, \begin{equation}
\begin{cases}
\widetilde{L}^{+}_{\omega}  y^{\mp}_1 = \mp e_\omega y^{\mp}_2  , \\
\widetilde{L}^{-}_{\omega}  y^{\mp}_2 = \pm e_\omega y^{\mp}_1 .
\end{cases}
\end{equation}
Recall that $ \widetilde{L}^{\pm}$ are self-adjoint operator.
\begin{align*}
    \im \int (i\; \widetilde{L}^{+}_{\omega} h_1 + \widetilde{L}^{-}_{\omega}h_2 )({y}_1^{\mp} + i {y}^{\mp}_2) \Psi dx &= \im \int i\,  ( \widetilde{L}^{+}_{\omega} h_1 ) {y}^{\mp}_1 \Psi + i\,  (\widetilde{L}^{-}_{\omega}h_2)   {y}^{\mp}_2 \Psi dx \\ & = \im \int i \, h_1   (\widetilde{L}^{+}_{\omega}  y^{\mp}_1 \Psi) + i \, h_2  (\widetilde{L}^{-}_{\omega}  y^{\mp}_2 \Psi) dx  \\ & = \im \int i \, h_1 (\mp e_\omega y^{\mp}_2 \Psi) + i \, h_2 ( \pm e_\omega y^{\mp}_1 \Psi) dx +O(e^{-2\delta \sqrt{\omega} \left| v \right| t }) \\ & = \mp e_\omega \im \int \overline{h}\,  \widetilde{\mathcal{Y}}^{\mp}_{\omega} \, \Psi dx +O(e^{-2\delta \sqrt{\omega} \left| v \right| t })  \\ & = \mp e_\omega \alpha^{\pm}(t,x) +O(e^{-2\delta \sqrt{\omega} \left| v \right| t })  .
\end{align*}
This concludes the proof of the Lemma \ref{equation-de-h+dy+da+} 
\end{proof}
\subsubsection{Control of the modulation parameters  }
We claim the following estimate of $v(t),\mu$ and $y$ on $[T(\alpha^{+}),T_n]$.
\begin{lem}[Control of $r,y$ and $\mu$.] 
\label{bosstrap-estimate}
For $T_0$ large enough independent of $n$ and $\forall \, \alpha^{+} $ such that  $$\left|\alpha^{+}\right| \leq e^{-\delta \sqrt{\omega} |v| T_n}. $$
the following holds
\begin{align}
\label{estimation-u-R-2}
\forall t \in [T(\alpha^{+}),T_n], \qquad \left\|u(t)-R(t)\right\|_{H^1_0} &\leq C e^{-\delta \sqrt \omega |v| t} \leq \frac{\epsilon}{2} \\
\label{estimation-on-v-2}
\left\|r(t)\right\|_{H^1_0} &\leq \frac{M}{2} e^{-\delta \sqrt \omega |v| t} \\
\label{estimation-on-mu-y-2}
\left| \mu(t) \right| + \left|y(t)\right| &\leq \frac{M'}{2} e^{-\delta \sqrt \omega |v| t } .
\end{align}

\end{lem}

We postpone the proof of Lemma $\ref{bosstrap-estimate}$ to the end of this section. 
\subsubsection{Control of the stable direction}
\begin{lem}
\label{conrol-of-a^-}
For $T_0$ large enough, independent of $n$ and $\forall \, \alpha^{+} $ such that $\left| \alpha^{+} \right| \leq e^{-\delta \sqrt\omega |v| T_n}$.\\
The following holds  \\
$$\forall t \in [T(\alpha^{+}),T_n] , \qquad \left|\alpha^{-}(t)\right| \leq \frac{1}{2} \; e^{-\delta \sqrt \omega |v| t}. $$
\end{lem}
\begin{proof}
$$\frac{d}{dt}\; (\alpha^{-}(t) e ^{-e_\omega t} )= (\frac{d}{dt}\alpha^{-}(t) -e_\omega \alpha^{-}(t) ) e ^{-e_\omega t} .$$ 
Due to \eqref{estimation-dt-alpha} and $\eqref{estimation-on-v-2}$, we have 
$$\left|\frac{d}{dt}\;(\alpha^{-}(t) e^{-e_\omega t} ) \right| \leq \left( C \frac{M^3}{8}   e^{ - 2 \delta \sqrt \omega |v| t } + C e^{-\delta \sqrt{\omega} \left| v \right| t}  \right) e^{-\delta \sqrt{\omega} \left| v \right| t } e^{-e_\omega t} .   $$
Then, we obtain by integration on $[t,T_n]$ and using that $\alpha^{-}(T_n)=0$, we get$$\left|\alpha^{-}(t)\right| \leq \left( C_3 \frac{M^3}{8}   e^{ - 2 \delta \sqrt \omega |v| t } + C_4 e^{-\delta \sqrt{\omega} \left| v \right| t}  \right) e^{-\delta \sqrt{\omega} \left| v \right| t }. $$
Hence, $$\forall t \in [T(\alpha^{+}),T_n], \qquad \left|\alpha^{-}(t)\right| \leq \frac{1}{2} e^{-\delta \sqrt \omega |v| t}.$$
If the following conditions are satisfied \begin{align}
    \label{3-condi-M}
   C_3 \frac{M^3}{8} e^{-2\delta \sqrt{\omega} \left| v \right| T_0}  \leq \frac{1}{4}, \\
   \label{1-cond-T_0}
   C_4 e^{-\delta \sqrt{\omega} \left| v \right| T_0 } \leq \frac{1}{4} .
\end{align}
\end{proof}
\subsubsection{Control of the unstable direction by a topological argument} 
Finally, we have to control~$\alpha^{+}(t)$. For this, we will provide the existence of a suitable value of $\alpha^{+}$.
\begin{lem}
For $\delta > 0$ small enough and $T_0$ large enough, there exists $\alpha^{+}$ such that \\ $\left| \alpha^{+} \right| \leq e^{-\delta \sqrt \omega |v| t }$  and   $T(\alpha^{+})=T_0.$
\end{lem}
\begin{proof}{We argue by contradiction.}\\
Assume that, $\forall \, \alpha^{+} $ such that 
$\left| \alpha^{+} \right| \leq e^{-\delta \sqrt \omega |v| t }$, one has $T(\alpha^{+})>T_0$.\\
From Lemma \ref{bosstrap-estimate} and \ref{conrol-of-a^-} we have 
\begin{align*}
    \left\|u(T(\alpha^{+})) - R(T(\alpha^{+})) \right\|_{H^1_0} \leq \frac{\varepsilon}{2} \\
    \left\| r (T(\alpha^{+})) \right\|_{H^1_0} \leq \frac{M}{2} e^{-\delta \sqrt \omega |v| T(\alpha^{+}) } \\
    \left| y(T(\alpha^{+})) \right| + \left| \mu(T(\alpha^{+})) \right| \leq \frac{M'}{2} e^{-\delta \sqrt \omega |v| T(\alpha^{+}) } \\
    \left| \alpha^{-}(T(\alpha^{+})) \right| \leq \frac{1}{2} e^{-\delta \sqrt \omega |v| T(\alpha^{+}) } .
 \end{align*}

By the definition of $T( \alpha^{+})$ and the continuity of the flow, one must have $$\left|          \alpha^{+}(T(\alpha^{+})) \right| = e^{-\delta \sqrt \omega |v| T( \alpha^{+})  }   . $$
Let $T<T(\alpha^{+})$ be close enough to $T(\alpha^{+})$ so that the solution $u(t)$ and its modulation are well-defined on $[T,T_n].$ \\
 For $t\in [T,T_n]$, let   $ \hspace{1.0 cm} \mathcal{N}(\alpha^{+}(t))=\mathcal{N}(t)= \left| e^{\delta \sqrt \omega |v| t } \alpha^{+}(t) \right|^{2} .$ \\
 \begin{equation}
\label{dt-N}
     \frac{d}{dt}\mathcal{N}(t) = e^{2\delta \sqrt \omega |v| t } \left[ 2\delta \sqrt \omega |v| \; \alpha^{+}(t) + 2 \frac{d}{dt} \alpha^{+}(t) \, \right] \alpha^{+}(t)
 \end{equation}
 Multiply by $ 2 \left| \alpha^{+}(t) \right| $ the estimate \eqref{estimation-dt-alpha}, we obtain 
 $$ \left| 2 \alpha^{+}(t) \frac{d}{dt} \alpha^{+}(t) + 2e_\omega \alpha^{+}(t)^2  \right| \leq C \left|\alpha^{+}(t) \right| \left( \left\| h(t) \right\|_{H^1_0}^3 + e^{-2\delta \sqrt \omega |v| t } \right) , $$
 which yieds $$  \frac{d}{dt} \left| \alpha(t)\right|^2 + 2e_\omega \left| \alpha(t) \right|^2\leq C \left|\alpha^{+}(t) \right| \left( \left\| h(t) \right\|_{H^1_0}^3 + e^{-2\delta \sqrt \omega |v| t } \right) $$ 
By \eqref{dt-N}, it follows that 
\begin{equation*}
\frac{d}{dt}\mathcal{N}(t) = e^{2\delta \sqrt \omega |v| t } [2 \delta \sqrt \omega |v| - 2e_\omega  ] | \alpha^{+}(t)|^{2} +  O \left(e^{2 \delta \sqrt \omega |v|t} \left| \alpha^{+}(t)\right| ( \left\| h(t) \right\|_{H^1_0}^3 + e^{-2\delta \sqrt \omega |v| t }) \right) .
\end{equation*}
Due to $\eqref{estimation-on-v-2}$ we have $$ e^{2 \delta \sqrt \omega |v| t } \left| \alpha^{+}(t) \right| ( \left\| h(t) \right\|_{H^1_0}^3 + e^{-2 \delta \sqrt \omega |v| t } ) \leq C  \sqrt{\mathcal{N}(t)} \left(  \frac{M^3}{8} e^{-2 \delta \sqrt{\omega} \left| v \right| t  }  + e^{- \delta \sqrt{\omega } \left| v \right| t  }\right) .$$
Let $\delta >0$ such that $2e_\omega - 2 \delta \sqrt \omega |v| \geq e_\omega,$ so that $$\frac{d}{dt}\mathcal{N}(t) \leq -e_\omega \mathcal{N}(t) +  \left( C_5 \frac{M^3}{8} e^{-2 \delta \sqrt{\omega} \left| v \right| t  }  + C_6 e^{- \delta \sqrt{\omega } \left| v \right| t  }\right) \sqrt{\mathcal{N}(t)}  . $$
We consider the above estimate at $t=T(\alpha^{+}) \geq T_0$, so large such that \begin{align}
    \label{4-cond-M}
    C_5\frac{M^3}{8} e^{- 2 \delta \sqrt{\omega} \left| v \right| T_0 } \leq \frac{1}{4} e_\omega , \\
    \label{2-cond-T_0}
    C_6 e^{-\delta \sqrt{\omega} \left|v \right| T_0 } \leq \frac{1}{4} e_\omega .
\end{align}

Using that $\mathcal{N}(T(\alpha^{+}))= 1$, we get 
\begin{equation}
\label{dN-dt}
 \forall \alpha^{+} \in B(e^{-\delta \sqrt \omega |v| T_n} ), \qquad  \frac{d}{dt}\mathcal{N}(T(\alpha^{+})) \leq  - \frac{1}{2} e_\omega.
\end{equation}
From \eqref{dN-dt}, a standard argument says that the map: $\alpha^{+} \longmapsto T(\alpha^{+})$ is continuous.\\

Indeed, by $\eqref{dN-dt}$, $\forall \varepsilon > 0, \, \exists \eta >0$ such that $$\mathcal{N}(T(\alpha^{+})-\varepsilon) > 1 + \eta , $$ and $$ \mathcal{N}(t) < 1 - \eta, \qquad \forall t \in [T(\alpha^{+})+ \varepsilon , T_n] \quad \text{(possibly empty)}.$$ By continuity of the flow of the (NLS) equation, it follows that $\exists \theta > 0$ such that, \\ for all $\left\| \widetilde{\alpha}^{+} - \alpha^{+} \right\| \leq \theta $, the corresponding $\widetilde{ \alpha }^{+}(t)$ satisfies  $$|\mathcal{N}(\widetilde{ \alpha }^{+}(t)) - \mathcal{N}( \alpha^{+}(t))| \leq \frac{\eta}{2}  \qquad \forall t \in [T(\alpha^{+})-\varepsilon,T_n] .$$
In particular, $T(\alpha^{+})- \varepsilon < T(\widetilde{\alpha}^{+}) < T(\alpha^{+}) + \varepsilon.$ \\
Now we consider the continuous map 
\begin{align*}
    P: B_{\R}(e^{-\delta \sqrt \omega |v| T_n})&  \longrightarrow S_{\R}(e^{-\delta \sqrt \omega |v| T_n}) \\
     \alpha^{+} & \longmapsto e^{-\delta \sqrt \omega |v| (T_n-T(\alpha^{+}))} \;
    \alpha^{+}(T(\alpha^{+}))
\end{align*}
Let $\alpha^{+} \in S_{\R}(e^{-\delta \sqrt  \omega |v| T_n})  $, from $\eqref{dN-dt}$ it follows that $T(\alpha^{+})=T_n$ and $P(\alpha^{+})=\alpha^{+}$, which means that $P|_{S_{\R}(e^{-\delta \sqrt \omega |v| T_n})}= Id $.
But this contradicts Brouwer's fixed point theorem.\\
So, $\exists \alpha^{+} \in B_{\R}(e^{-\delta \sqrt \omega |v| T_n})$ such that $T(\alpha^{+})=T_0$.
\end{proof}
\subsection{Estimate on the modulation parameters}
\label{Estimate on the modulation parameters}
\begin{proof} This section is devoted to the proof of the Lemma \ref{bosstrap-estimate}.
For that, we claim the following results which will be proved at the end of the proof.\\
Let us recall that  $\widetilde{R}(t,x)=e^{i \widetilde{\varphi} (t,x)} \widetilde{Q}_{\omega}(t,x) \Psi(x) .$
\begin{clm}
\label{claim-dt-E(r)}
\begin{equation}
\label{d-dt-E(r)} 
\left| \frac{d}{dt} \left( E(\widetilde{R}(t))+(\frac{\omega}{2} + \frac{|v|^2}{8} )M(\widetilde{R}(t)) - \frac{v}{2} P(\widetilde{R}(t)) \right)    \right| \leq C e^{-2\delta \sqrt \omega |v| t} + M^2e^{-3\delta \sqrt{\omega} \left| v \right| t } .
\end{equation}
\end{clm}
\begin{clm}
\label{clm-E(u)-E(r)}
\begin{multline}
\label{E(u)-E(r)}
\bigg| \left[ E(u(t)) + (\frac{\omega}{2}+\frac{\left|v\right |^2}{8}) M(u(t))-  \frac{v}{2}P(u(t)) \right] -  \left[ E(\widetilde{R}(t))+ (\frac{\omega}{2} + \frac{|v|^2}{8})M(\widetilde{R}(t)) - \frac{v}{2}P(\widetilde{R}(t)) \right] \\ - \frac{1}{2} \left[(\widetilde{L}^{+}_{\omega}h_1(t),h_1(t))+(\widetilde{L}^{-}_{\omega}h_2(t),h_2(t)) \right]   \bigg| \leq C M e^{-2\delta \sqrt \omega |v| t } + C M^2 e^{-3\delta \sqrt{\omega} \left|v \right| t} . 
\end{multline}
\end{clm}

\begin{clm}
\label{estima-v-3}
There exists $ C >0 $ such that,
\begin{equation}
\label{trans-coercivity}
\left\|h(t)\right\|_{H^1_0}^2 \leq  C \bigg[\left(\widetilde{L}^{+}_{\omega}h_1(t),h_1(t)\right)+ \left(\widetilde{L}^{-}_{\omega}h_2(t),h_2(t)\right)+  \left(\alpha^{\pm}(t)\right)^2 + M^2 e^{-4\delta \sqrt{\omega}|v| t}\bigg]
\end{equation}
\end{clm}

Now, we start the proof of Lemma \ref{bosstrap-estimate}. Let $t\in [T(\alpha^{+}),T_n]$, integrating $\eqref{d-dt-E(r)}$ on $[t,T_n]$ we~get 
\begin{multline*}
\bigg| \left[E(\widetilde{R}(T_n))+(\frac{\omega}{2}+\frac{|v|^2}{8}) M(\widetilde{R}(T_n))-\frac{v}{2}P(\widetilde{R}(T_n))\right]  -\left[ E(\widetilde{R}(t)) +(\frac{\omega}{2}+\frac{|v|^2}{8}) M(\widetilde{R}(t)) - \frac{v}{2} P(\widetilde{R}(t))  \right]  \bigg| \\ \leq C e^{-2\delta \sqrt \omega |v| t} + M^2e^{-3\delta \sqrt{\omega} \left| v \right| t } .
\end{multline*} 
From the above estimate and  $\eqref{E(u)-E(r)}$, we have  
\begin{multline}
\label{estim-on-L(T_n)-L(t)}
\left| \bigg[ \left(\widetilde{L}^{+}_{\omega}h_1(T_n),h_1(T_n) \right)   + \left(\widetilde{L}^{-}_{\omega}h_2(T_n),h_2(T_n)  \right) \bigg]  - \bigg[ \left( \widetilde{L}^{+}_{\omega}h_1(t),h_1(t) \right) + \left(\widetilde{L}^{-}_{\omega}h_2(t),h_2(t) \right) \bigg]\right| \\  \leq C M e^{-2\delta \sqrt \omega |v| t} + C M^2 e^{-3 \delta \sqrt \omega |v| t}. 
\end{multline}
From Lemma \ref{modulation} and Lemma \ref{modulatedfinaldata} we have 
\begin{equation}
\label{estim-on-L(T_n)}
\left| \left(\widetilde{L}^{+}_{\omega} h_1(T_n),h_1(T_n) \right)   + \left(\widetilde{L}^{-}_{\omega}h_2(T_n),h_2(T_n)  \right) \right| \leq C \left\|h(T_n) \right\|_{H^1_0}^2 \leq C \left| \lambda \right|^2
\leq C e^{-2 \delta \sqrt \omega |v| t} .
\end{equation} 

We deduce from \eqref{estim-on-L(T_n)-L(t)}, \eqref{estim-on-L(T_n)} and the Claim \ref{estima-v-3} that  
\begin{align*}
\left\|h(t)\right\|_{H^1_0}^2 &\leq C (\widetilde{L}^{+}_{\omega}h_1(t),h_1(t)) + C (\widetilde{L}^{-}_{\omega}h_2(t),h_2(t)) + C \left(\alpha^{\pm}(t)\right)^2 + CM^2e^{-4\delta \sqrt{\omega} |v| t} \\ &\leq C_7 M e^{-2\delta \sqrt \omega |v| t} + C_8 M^2 e^{-3 \delta \sqrt \omega |v| t} .
\end{align*}

If $T_0$ satisfies \begin{align}
\label{5-condi-M}
C_7e^{-2\delta \sqrt{\omega } \left| v \right| T_0} \leq \frac{1}{4},   \\
\label{3-cond-T_0}
C_8 M e^{-3\delta \sqrt{\omega } \left| v \right| T_0} \leq \frac{1}{4}.
\end{align}

Then, we have provide $$ \left\|h(t)\right\|_{H^1_0} \leq \frac{M}{2} e^{- \delta \sqrt \omega |v| t }.$$
If conditions \eqref{1+2-condi-M}, \eqref{3-condi-M}, \eqref{1-cond-T_0},  \eqref{2-cond-T_0}, \eqref{4-cond-M}, \eqref{5-condi-M} and \eqref{3-cond-T_0} on $M$ and $T_0$ hold. However it is easy to find $T_0$ and $M$ satisfying these conditions. We take $T_0$ large enough such that
\begin{equation}
\max(C_4,C_6,C_7) e^{-\delta \sqrt{\omega} \left| v \right|  T_0 } \leq \frac{1}{4}\min(1,e_\omega), 
\end{equation}
and we take $M$ such that
\begin{align}
 M e^{-\delta \sqrt{\omega}  \left|  v \right| T_0  }  &\leq \frac{1}{10 C^{'}} \min{ \left( \left\| \partial_{x_j} Q \Psi  \right\|^2_{L^2},  \left\| Q \Psi \right\|^2_{L^2} \right)}, \\
\max(  C_3, C_5)  \frac{M^3}{8} e^{-2\delta \sqrt{\omega }  \left| v \right| T_0 } &\leq  \frac{1}{4}  \min(1,e_\omega).  \\
C_8 M e^{-\delta \sqrt{\omega} \left| v \right| T_0} &\leq \frac{1}{4}
\end{align}

From Lemma $\ref{equation-de-h+dy+da+}$  we have\begin{align*}
    \left|\frac{d\mu(t)}{dt}\right|+\left|\frac{dy(t)}{dt}\right| &\leq C \left\|h(t) \right\|_{H^1_0}^2 + Ce^{-2\delta \sqrt{\omega} |v| t } \\ &\leq C \frac{M^2}{4}e^{-2\delta \sqrt{\omega} \left|v\right| t} + C e^{-2 \delta \sqrt{\omega} \left|v\right| t } .
\end{align*}
We integrate the above estimate on some time interval $[t,T_n]$, for $t \in [T(\alpha^{+}),T_n]$.
$$\left| \mu(t) \right|+\left|y(t)\right| \leq \left|\mu(T_n) \right|+ \left|y(T_n)\right| + C \frac{M^2}{4} e^{-2\delta \sqrt{\omega} \left|v\right| t } + C e^{-2\delta \sqrt{\omega} |v| t }. $$ 
\vspace{-0.05cm}
Furthermore, due to the definition of $T(\alpha^{+})$ we get  $$\left| \mu(t) \right| + \left| y(t) \right| \leq C'_{1} e^{-2 \delta \sqrt \omega |v| t } + C'_{2} \frac{M^2}{4} e^{-2\delta \sqrt{\omega} \left|v\right| t } . 
 $$
 Then, we can deduce that $\qquad \left|\mu(t)\right|+ \left|y(t)\right| \leq \frac{M'}{2} e^{-\delta \sqrt \omega |v| t}. $ \\ 

Provided, for $T_0$ large enough
\begin{equation}
\label{condi-T_0-M'}
   C'_{1} e^{-\delta \sqrt{\omega}\left|v\right| T_0 } \leq \frac{M'}{4} ,
 \end{equation}
 
 and we take $M'$ such that
\begin{equation}
   \label{Condi-M-M'}
    C'_{2} \frac{M^2}{4}  e^{-2 \delta \sqrt{\omega} \left| v \right| T_0}    \leq \frac{M'}{4}.
\end{equation}

Finally, we obtain  \begin{align*}
    \left\| u(t)-R(t) \right\|_{H^1_0} &\leq \left\|R(t)-\widetilde{R}
(t)\right\|_{H^1_0}+ \left\| h(t) \right\|_{H^1_0} \\ &\leq C \left|y(t)\right| + \left\|h(t)\right\|_{H^1_0} \\ &\leq C e^{-\delta \sqrt \omega |v| t} \\ &\leq \frac{\varepsilon}{2},
\end{align*}
which concludes the proof of Lemma \ref{bosstrap-estimate}, by taking $T_0$ large enough. 
\end{proof}
\begin{proof}[Proof of Claim $\ref{claim-dt-E(r)}$.]
Recall that  $\widetilde{R}(t,x)=e^{i\widetilde{\varphi}(t,x)} \widetilde{Q}_{\omega}(t,x)\Psi(x).$

\begin{equation*}
\nabla \widetilde{R}(t,x) = [i \frac{v}{2} \widetilde{Q}_{\omega} \Psi + \nabla(\widetilde{Q}_{\omega}\Psi) ] \; e^{i \widetilde{\varphi}(t,x)}, \qquad 
\left| \nabla \widetilde{R}(t,x) \right|^2 = \frac{|v|^2}{4} \widetilde{Q}_{\omega}^2 \Psi^2 + \left| \nabla(\widetilde{Q}_{\omega} \Psi)\right|^2.
\end{equation*}
\begin{align*}
\displaystyle E(\widetilde{R}(t))&= \frac{1}{2} \displaystyle\int \left| \nabla \widetilde{R}(t) \right|^ 2 dx - \frac{1}{p+1} \int \widetilde{Q}_{\omega}^{p+1} \Psi^{p+1} dx, \\ 
\displaystyle\frac{d}{dt} E(\widetilde{R}(t)) &= \frac{1}{2} \frac{d}{dt} \left[ \int \frac{|v|^2}{4} \widetilde{Q}_{\omega}^2 \Psi^2 + \left|\nabla(\widetilde{Q}_{\omega}\Psi)\right|^2 dx - \frac{1}{p+1} \int \widetilde{Q}_{\omega}^{p+1} \Psi^{p+1} dx \right] \\
 &= \frac{1}{2}  \int  \frac{|v|^2}{4}2(-v-\frac{dy(t)}
{dt}) \nabla \widetilde{Q}_{\omega} \, \widetilde{Q}_{\omega} \Psi^2 + 2 (-v-\frac{dy(t)}{dt}). \nabla (\nabla (\widetilde{Q}_{\omega} \Psi)) \nabla(\widetilde{Q}_{\omega}\Psi) dx\\  & -  \frac{1}{p+1} \int (p+1) (-v-\frac{dy(t)}{dt}) \nabla \widetilde{Q}_{\omega} \, \widetilde{Q}_{\omega}^p \Psi^{p+1 }  dx \\ 
 & = (-v-\frac{dy(t)}{dt}) \left[ \int \frac{|v|^2}{4}  \nabla \widetilde{Q}_{\omega}\, \widetilde{Q}_{\omega} \, \Psi^2 + \nabla(\nabla \widetilde{Q}_{\omega} \Psi) \nabla(\widetilde{Q}_{\omega}\Psi) dx - \int \nabla \widetilde{Q}_{\omega}  \, \widetilde{Q}_{\omega}^{p} \, \Psi^{p+1}  \,  \right],  \\
 \text{ where, \; } &(-v-\frac{dy(t)}{dt}).\nabla((\nabla \widetilde{Q}_{\omega} \Psi)) \nabla(\widetilde{Q}_{\omega}\Psi)= \sum_{k=1}^3 \sum_{j=1}^3 \left(-v_k-\frac{dy_k(t)}{dt}\right) \partial_{x_k}\partial_{x_j} (\widetilde{Q}_{\omega} \Psi)  \partial_{x_j}(\widetilde{Q}_{\omega} \Psi). \\ 
 M(\widetilde{R}(t))&=\displaystyle \int \left|\widetilde{R}(t)\right|^2 dx , \\
 \displaystyle\frac{d}{dt} M(\widetilde{R}(t))&= \frac{d}{dt} \displaystyle\int  \widetilde{Q}_{\omega} ^ 2\Psi^2 dx = 2 (-v -\frac{dy(t)}{dt}) \int  \nabla \widetilde{Q}_{\omega} \, \widetilde{Q}_{\omega} \, \Psi^2 dx .\\   
\displaystyle P(\widetilde{R}(t))&= \im \int \nabla \widetilde{R}(t) \overline{\widetilde{R}}(t) dx  ,\\  \displaystyle\frac{d}{dt} P(\widetilde{R}(t)) &=
\frac{d}{dt} \left(\frac{v}{2} \int  \widetilde{Q}_{\omega}^2 \, \Psi^2 dx \right) =  v \int (-v-\frac{dy(t)}{dt} ) \nabla \widetilde{Q}_{\omega} \, \widetilde{Q}_{\omega} \Psi^2 dx.
\end{align*}
Hence, we have
\begin{multline*}
    \frac{d}{dt} \left[ E(\widetilde{R}(t)) + ( \frac{\omega}{2} + \frac{|v|^2}{8})M(\widetilde{R}(t)) -\frac{v}{2} P(\widetilde{R}(t))\right]= \frac{\omega}{2} (-v-\frac{dy(t)}{dt}) \int \nabla \widetilde{Q}_\omega \, \widetilde{Q}_\omega \Psi^2 \\  +(-v-\frac{dy(t)}{dt}) \left[ \int \nabla ( \nabla \widetilde{Q}_{\omega} \, \Psi) \nabla(\widetilde{Q}_{\omega} \Psi) dx - \int \nabla \widetilde{Q}_{\omega} \, \widetilde{Q}_{\omega}^p \, \Psi^{p+1} \right] .
\end{multline*}
For the first integral, we have\begin{align*}
   \left| \int \nabla \widetilde{Q}_\omega \, \widetilde{Q}_\omega  \Psi^2 \right|= \left| \frac{1}{2} \int \nabla \widetilde{Q}_\omega^2 \, \Psi^2 \right| = \left| \int \widetilde{Q}_\omega^2  \nabla \Psi \, \Psi \right| \leq C \, e^{-2 \delta \sqrt{\omega} \left|v\right| t  } .
\end{align*}

Using $\eqref{estimation-en-mu-et-y}$ and the fact that the support of the derivatives of $\Psi$ is compact. Furthermore, in the second integral, we have some terms with $\Psi$ witch doesn't have a compact support. For this terms, we have to use the fact that $Q_\omega$ is a radial function, concluding the proof of the Claim \ref{claim-dt-E(r)}
\end{proof}

\begin{proof}[Proof of Claim \ref{clm-E(u)-E(r)}.]
Recall that,$$u(t,x)=e^{i\widetilde{ \varphi }(t,x)}\left(\widetilde{Q}_{\omega}(t,x)\Psi(x)+h(t,x)\right).$$ \begin{align*}
E(u(t)) &=E(e^{i\widetilde{ \varphi }}\left[\widetilde{Q}_{\omega}\Psi+h\right]) \\ &= \frac{1}{2} \int_{\Omega} \left| \nabla \left(  e^{i\widetilde{ \varphi }}\left(\widetilde{Q}_{\omega}\Psi+h\right) \right) \right|^2 dx - \frac{1}{p+1} \int_{\Omega} \left |\widetilde{Q}_{\omega} \Psi + h \right|^{p+1} dx .
\end{align*}
Using Taylor expansion,
\begin{align*}
    \left| \widetilde{Q}_{\omega}\Psi + h \right|^{p+1}&= \widetilde{Q}_{\omega}^{p+1} \,\Psi^{p+1} + \left(\frac{p+1}{2}\right) \widetilde{Q}_{\omega}^{p} \, \Psi^{p}\, (h + \overline{h} )  \\ &+\frac{1}{2} \left(\frac{p+1}{2}\right)\left(\frac{p-1}{2}\right) \, \widetilde{Q}_{\omega}^{p-1} \, \Psi^{p-1}\, \left( h^ 2+ \overline{h}^2 \right)  +  \left(\frac{p+1}{2}\right)^2 \,\widetilde{Q}_{\omega}^{p-1} \, \Psi^{p-1} h \, \overline{h} + \beta(t,x) .
\end{align*}
and 
\begin{align*}
\left| \nabla \left( e^{i\widetilde{ \varphi }}(\widetilde{Q}_{\omega}\Psi + h)   \right)  \right|^2 &= \left| e^{i\widetilde{ \varphi }} \left(
i \frac{v}{2} (\widetilde{Q}_{\omega} \Psi+h) + (\nabla ( \widetilde{Q}_{\omega} \, \Psi)  + \nabla h  ) \right) \right|^2 \\ 
& =  \frac{|v|^2}{4} \left|\widetilde{Q}_{\omega} \Psi+ h \right|^{2}  -  v \, \nabla(\widetilde{Q}_{\omega} \Psi) h_2 + v \,  \widetilde{Q}_{\omega} \Psi \nabla h_2 + v (h_1 \nabla h_2 - h_2 \nabla h_1) \\ &+ \left| \nabla (\widetilde{Q}_{\omega} \Psi) \right|^2  + 2 \nabla (\widetilde{Q}_{\omega} \Psi) \nabla h_1 + \left|\nabla h\right|^2.
\end{align*}
Here and until the end the proof:  $\int $ denote the integral over $\Omega$. \\ 

We have
\begin{align*}
E(u(t))- E(\widetilde{R}(t)) &=  
\frac{\left|v\right|^2}{4}  \int \widetilde{Q}_{\omega} \Psi h_1 +  \frac{\left|v \right|^2}{8} \int \left|h\right|^2+ \frac{1}{2} \int  \left| \nabla h \right|^2 +  \int \nabla (\widetilde{Q}_{\omega} \Psi).\nabla h_1 - \int \widetilde{Q}_{\omega}^p \Psi^p h_1 \\ &  - \int  v.\nabla(\widetilde{Q}_{\omega} \Psi) h_2 + \int \frac{v}{2}.(h_1 \nabla h_2 - h_2 \nabla h_1) - \frac{p}{2} \int \widetilde{Q}_{\omega}^{p-1} \Psi^{p-1} h_1^2 \\& + \frac{1}{2} \int \widetilde{Q}_{\omega}^{p-1} \Psi^{p-1} h_2^2  + \beta(t,x) .
\end{align*}
\begin{align*}
\left( \frac{\omega}{2} + \frac{\left|v \right|^2}{8} \right) \left(  M(u(t))- M(\widetilde{R}(t)) \right) &=  \left( \omega + \frac{ \left| v \right|^2}{4} \right) \int \widetilde{Q}_{\omega} \Psi h_1 + \left( \frac{\omega}{2}  + \frac{\left| v \right|^2}{8} \right) \int \left| h \right|^2    .
\end{align*}
\begin{align*}
-\frac{v}{2}. \left( P(u(t))-P(\widetilde{R}(t)) \right)   &=   -\frac{\left| v \right|^2}{2} \int \widetilde{Q}_{\omega} \Psi h_1 - \frac{\left| v \right|^2}{4} \int \left| h \right|^2  - \int \frac{v}{2}. \left(h_1\nabla h_2 + h_2 \nabla h_1 \right) \\ &+ \int v.\nabla (\widetilde{Q}_{\omega} \Psi ) h_2 .
\end{align*}

Then we have, 
\begin{align*}
& \left[ E(u(t)) + (\frac{\omega}{2}+\frac{|v|^2}{8}) M(u(t))- \frac{v}{2}P(u(t)) \right] -  \left[ E(\widetilde{R}(t))+ (\frac{\omega}{2}+\frac{|v|^2}{8})M(\widetilde{R}(t)) - \frac{v}{2}P(\widetilde{R}(t)) \right]  \\ &= \frac{1}{2}\left[ (\widetilde{L}^{+}_{\omega}h_1,h_1) + (\widetilde{L}^{-}_{\omega}h_2,h_2) \right] -\frac{p}{2} \int \widetilde{Q}_{\omega}^{p-1} h_1^2 (\Psi^{p-1}-1) - \frac{1}{2} \int \widetilde{Q}_{\omega}^{p-1} h_2^2(\Psi^{p-1} -1 )\\ &  + \int -\Delta(\widetilde{Q}_{\omega} \Psi) h_1 dx - \int \widetilde{Q}_{\omega}^{p} \Psi^{p} h_1 + \int  \omega \widetilde{Q}_{\omega} \Psi h_1 + \beta(t,x) \\ & = \frac{1}{2}\left[ (\widetilde{L}^{+}_{\omega}h_1,h_1) + (\widetilde{L}^{-}_{\omega}h_2,h_2) \right] - \frac{p}{2} \int \widetilde{Q}_{\omega}^{p-1} h_1^2 (\Psi^{p-1}-1) - \frac{1}{2} \int \widetilde{Q}_{\omega}^{p-1} h_2^2(\Psi^{p-1} -1 ) \\
&+ \int \underbrace{(-\Delta \widetilde{Q}_{\omega} + \omega \widetilde{Q}_{\omega} -\widetilde{Q}_{\omega}^{p} ) }_{=0} \Psi  \,h_1 - 2 \int \nabla \widetilde{Q}_{\omega} \nabla \Psi h_1  - \int \widetilde{Q}_{\omega} \Delta \Psi h_1  - \int \widetilde{Q}_{\omega}^{p}\Psi (\Psi^{p-1}-1) h_1 +\beta(t,x).
\end{align*}

Using the fact that $\nabla \Psi,\Delta \Psi$ and $(\Psi^{p-1} - 1)$ has a compact support, to conclude the proof of Claim  \ref{clm-E(u)-E(r)}. 
\end{proof}
\begin{proof}[Proof of Claim \ref{estima-v-3}.]
The proof of \eqref{trans-coercivity} is a standard consequence of Lemma \ref{spectalprop} and the following orthogonality conditions, $
\re \displaystyle\int \partial_{x_j} \widetilde{Q}_{\omega}\,  \Psi \, \overline{h} \,dx = 0 ,
\im \displaystyle\int \widetilde{Q}_{\omega} \, \Psi  \, \overline{h} \,  dx = 0 .$ \\

Due to \eqref{estimation-v-1}, there exits $C>0$ such that  
\begin{align*}
\left\|h(t)\right\|_{H^1_0}^2 \leq  C \bigg[&\left(\widetilde{L}^{+}_{\omega}h_1(t,x),h_1(t,x)\right)+ \left(\widetilde{L}^{-}_{\omega}h_2(t,x),h_2(t,x)\right)+ \sum_{j=1}^3\left(\int \partial_{x_i} \widetilde{Q}_{\omega}(t,x) h_1(t,x) dx \right)^2 \\ &+  \left(\int \widetilde{Q}_{\omega}(t,x) h_2(t,x) dx \right)^2 + \left( \im \int  \widetilde{\mathcal{Y}}^{\mp}_{\omega}(t,x) \,  \overline{h}(t,x)  dx \right)^2 \bigg]
\end{align*}
Using the orthogonality conditions, we get $$\int \partial_{x_j} \widetilde{Q}_{\omega} h_1= \int \partial_{x_j} \widetilde{Q}_{\omega} (1- \Psi) h_1  \quad  \text{ and } \quad \int \widetilde{Q}_{\omega} h_2 = \int \widetilde{Q} (1-\Psi) h_2 \, .$$  Due to the exponential decay of $Q$ and the compact support of $(1-\Psi)$ , we have   $$   \left| \int \partial_{x_j} \widetilde{Q}_{\omega}(t) h_1(t) \right| \leq C M 
e^{-2 \delta \sqrt \omega |v| t}  \quad \text{ and } \quad \left| \int \widetilde{Q}_{\omega}(t) h_2(t)  \right| \leq C M e^{-2 \delta \sqrt \omega |v| t } $$ 
\begin{align*}
\im \int  \widetilde{\mathcal{Y}}^{\mp}_{\omega}(t,x) \,  \overline{h}(t,x) &= \alpha^{\pm}(t) +  \im \int  \widetilde{\mathcal{Y}}^{\mp}_{\omega}(t,x) \,  \overline{h}(t,x)(1-\Psi(x)) dx   \\ &= \alpha^{\pm}(t) + O\left(M e^{-2\delta \sqrt{\omega} |v| t } \right)
\end{align*}
This concludes the proof of the Claim \ref{estima-v-3}.

\end{proof}
\begin{section}{Fixed point theorem}
\label{fixed point theorem}
\begin{proof}{This section is devoted to the proof of Theorem $\ref{theoremWithVelocityGrand}.$}

Recall that, if $\Theta = \emptyset $ then
$H(t,x)=e^{i\varphi(t,x)}Q_\omega(x-tv)$, where $\varphi(t,x)=\frac{1}{2}(x.v)-\frac{1}{4} \left| v \right|^2 t + t\omega$, is an exact soliton solution of~(NLS). \\ 

Let  $R(t,x)= e^{i\varphi(t,x)} Q_{\omega}(x-tv) \Psi(x)$. Write $$(i\partial_t  + \Delta )  R = - \Psi \left| H  \right|^2 H+ 2 \nabla \Psi \nabla H+ \Delta \Psi H. $$
 We look for $r_\omega \in C([T_0,+\infty),H^2(\Omega) \cap H^1_0(\Omega))$ such that
 \begin{equation}
 \label{equa-on-p}
     \begin{cases}
     i\partial_t r_\omega + \Delta r_\omega = -\left| R + r_\omega 
     \right|^2 \left( R + r_\omega \right) + \Psi \left| H  \right|^2 H - 2 \nabla \Psi \nabla H  - \Delta \Psi H , \\
     r_\omega(t) \longrightarrow 0 \; \text{ as  } \;t \longrightarrow +\infty \; \text{ in } H^2(\Omega) \cap H^1_0(\Omega). 
     \end{cases}
 \end{equation}
Set
\begin{align*}
A_0(t,x)&= \Psi(x) (1-\Psi^2(x)) \left| H(t,x)\right|^2 H (t,x) - 2 \nabla \Psi(x) \nabla H(t,x)- \Delta \Psi(x) H(t,x) , \\ 
A_1(r_\omega(t,x))&= - R(t,x)^2\overline{r_\omega}(t,x)- 2 \left| R(t,x) \right|^2 r_\omega(t,x) , \\
A_2(r_\omega(t,x))&= - \overline{R}(t,x) r_\omega^2(t,x) - 2  R(t,x)  \left| r_\omega(t,x) \right|^2 , \\
A_3(r_\omega(t,x))&= - \left| r_\omega(t,x) \right|^2 r_\omega (t,x).
\end{align*}

We shall look for solutions of \eqref{equa-on-p} in this space: 
$$E=\{ r_\omega \in C\left([T_0,+\infty),H^2(\Omega) \cap H^1_0(\Omega)\right), \;
\left\|r_\omega \right\|_E < \infty       \}, $$
such that $$ \left\| r_\omega \right\|_E = \sup_{ t\geq T_0} \left\{ e^{\delta \sqrt{\omega} \left| v \right| t }\left( \frac{1}{\left|v\right|^3} \left\|r_\omega \right\|_{H^2(\Omega)} + \left\| r_\omega \right\|_{L^2(\Omega)} \right) \right\} . $$
Let \begin{equation*} \begin{aligned}   
    \Phi:(B_E,d_E) &\longrightarrow (B_E,d_E)  \\
    r_\omega \longmapsto  & \Phi(r_\omega)= -i \int_{t}^{+\infty} S(t-\tau)  A_0(\tau) d\tau - i \sum_{k=1}^{3} \int_{t}^{+\infty} S(t-\tau)  A_k(r_\omega(\tau)) d\tau .
\end{aligned}
\end{equation*}
Where  $ B_E=B_E(0 , 1)=\{ h \in E, \;  \left\| h\right\|_E \leq 1  \}$ and $\displaystyle d_E(h,g) =  \left\|h-g\right\|_{E} .$  \\
   One can check that $(B_E,d_E)$ is a complete metric space.\\
   Here S(t) is the unitary group of the linear Schr\"odinger equation with Dirichlet boundary conditions.\\
    
Denote,\begin{align*}
 \displaystyle   \rm{J}_0(t)&=  \displaystyle  \int_t^{+\infty} S(t-\tau) A_0(\tau) d\tau , \\
 \displaystyle   \rm{J}_k(r_\omega(t))&=  \displaystyle \int_t^{+\infty} S(t-\tau) A_k(r_\omega(\tau)) d\tau, \quad k=1,2,3. 
     \end{align*}
     
  \begin{rem}
  \label{proof-for-any-p}
 For $2\leq p<5$, the proof is also based on a fixed point theorem as the cubic case. Indeed, we have to use Taylor expansion for the non-linearity $\left| R+r_\omega \right|^{p-1} \left( R+ r_\omega \right)$ and we can divid the function $\Phi$ in three integrals, one for the constant terms on $r_\omega$, the other for the linear terms on $r_\omega$ and the last one for the nonlinear terms on $r_\omega$. Finally, we use the same space $E$ and norm to prove that $\Phi$ is a contraction mapping for high velocity.\\ 
 \end{rem}
 
In step $1$, we will prove that the ball $B_E$ is stable by $\Phi$ and in the second step we will prove that $\Phi$ is a contraction mapping on the complete metric space $(B_E,d)$. Finally, in step $3$ we will conclude by fixed point theorem the existence of the solution of the (NLS$_\Omega$).
\begin{itemize}
\item Step 1 :   Stability of $B_E$ by $\Phi$. 
\label{Step1}
\begin{lem}
\label{lemmaEstimationOn_J_K}
There exists  $C_{\omega}>0$ and $\delta>0$ such that,  \begin{align}
\label{estimation de J_0}
\left\| \rm{J}_0 \right\|_{E} &\leq \frac{C_\omega}{  \left| v \right|} \\
 \label{EstimationJ_1}
\left\| \rm{J}_1(r_\omega) \right\|_{E} &\leq \frac{C_\omega}{ \left| v \right|  } \left\|r_\omega \right\|_{E}  \\
 \label{EstimationJ_2}
 \left\| \rm{J}_2(r_\omega) \right\|_{E} &\leq C_{\omega} \left| v \right|^4  e^{-\delta \sqrt{\omega} \left| v\right| T_0} \left\|r_\omega \right\|_{E}^2  \\
 \label{EstimationJ_3}
 \left\| \rm{J}_3(r_\omega) \right\|_{E} &\leq C_{\omega}\left| v \right|^5  e^{-2\delta \sqrt{\omega} \left| v\right| T_0} \left\|r_\omega \right\|_{E}^3  \\
 \forall r_\omega \in B_E, \quad  \left\| \Phi(r_\omega) \right\|_E & \leq 1.
\end{align}
\end{lem}

\begin{proof}
\begin{enumerate}
    \item  \textbf{Estimate For $J_0$.} \\ 
Recall that $ A_0(t,x)= \Psi(x) (1-\Psi^2(x)) \left| H(t,x)\right|^2 H (t,x) - 2 \nabla \Psi(x) \nabla H(t,x)- ~\Delta \Psi(x) H(t,x),$ 
where
$H(t,x)=Q_{\omega}(x-tv)e^{i(\frac{1}{2}(x.v)- \frac{ \left|v\right|^2}{4}t+ t\, \omega )}.$ \\

Let us prove that there exists $C_\omega > 0$ such that,\begin{equation}
\label{EquationDeA_0}
    \left\| A_0(t) \right\|_{H^2} \leq C_\omega \left| v \right|^3  e^{-\delta \sqrt{\omega} \left| v \right| t  }, \quad \forall t \in [T_0,+\infty).
 \end{equation}
 It suffices to estimate the $L^2$ norm of $A_0$ and $\nabla^2 A_0$, due to the following elementary interpolation inequality, if $f \in H^2$, 
\begin{equation}
\label{interpolation-inequality}
    \left\| \nabla f  \right\|_{L^2} \leq  \left\| \nabla^2 f \right\|_{L^2}^{\frac{1}{2}} \left\| f \right\|_{L^2}^{\frac{1}{2}} .
\end{equation}

We will use the fact that $\Psi(1-\Psi^2)$, $\nabla \Psi$ and $\Delta \Psi $ have a compact support. We~will suppose that their support is include in $ \{ \left| x \right| < M \}$, for some $M>0$.  \\ 

Let $x \in \text{supp}\left\{ \Psi(1-\Psi^2) \right\} \subset \left\{ \left| x \right| < M  \right\}$ then 
$ \left\{  \; t \left| v \right| - M \leq \left| x -tv \right|  \; \right\}.$ \\ 

By \eqref{eq_QQ}, we have 
\begin{equation}
    \begin{cases}
    \label{Supp-Compact-Q}
    \left| Q_\omega(x-tv) \right| \leq C_\omega e^{\delta \sqrt{\omega} M } e^{-\delta \sqrt{\omega} \left| v \right| t } ,  \\
    \left|\nabla Q_\omega(x-tv)\right| \leq C_\omega e^{\delta \sqrt{\omega} M } e^{-\delta \sqrt{\omega} \left| v \right| t } . 
    \end{cases}
\end{equation}
Then, \begin{equation*}
    \left\| A_0 \right\|_{L^2} \leq C_\omega \left| v \right|   e^{-\delta \sqrt{\omega} \left| v \right| t } . 
\end{equation*}

Now, let us estimate $\nabla^2 A_0.$  

Recall that $ \displaystyle A_0=  \Psi (1-\Psi^2) \left| H\right|^2 H  - 2 \sum_{k=1}^{3} \partial_{x_k} \Psi \partial_{x_k} H- ~\Delta \Psi H. $

\begin{align*}
 \partial_{x_j} \partial_{x_i} A_0(t,x)&=   \partial_{x_j} \partial_{x_i} \left[ \Psi (1-\Psi^2) \right] \left| H\right|^2 H  + \partial_{x_i} \left[ \Psi (1-\Psi^2) \right]  \partial_{x_j}  \left[ \left| H\right|^2 H  \right] \\& +   \partial_{x_j} \left[ \Psi (1-\Psi^2) \right]  \partial_{x_i}  \left[ \left| H\right|^2 H\right] + \left[ \Psi (1-\Psi^2) \right] \partial_{x_j}   \partial_{x_i} \left[ \left| H\right|^2 H\right] \\
&-2  \left(\sum_{k=1}^{3}  \partial_{x_j} \partial_{x_i} \left[\partial_{x_k} \Psi \right]  \;  \partial_{x_k} H + \partial_{x_i} \left[\partial_{x_k} \Psi \right] \partial_{x_j} \partial_{x_k} H   \right) \\ 
&-2 \left( \sum_{k=1}^{3} \partial_{x_j}\partial_{x_k} \Psi \; \partial_{x_i} \left[ \partial_{x_k} H \right] + \partial_{x_k} \Psi \;  \, \partial_{x_j} \partial_{x_i} \left[ \partial_{x_k} H \right] \right)\\
& -   \partial_{x_j} \partial_{x_i} \left[ \Delta \Psi \right] H -  \partial_{x_i} \left[ \Delta \Psi \right] \partial_{x_j} H- \partial_{x_j}\left[ \Delta \Psi \right] \partial_{x_i}H- \left[ \Delta \Psi \right] \partial_{x_j} \partial_{x_i}H
\end{align*} 

\begin{clm}
\label{claim calcule}
\begin{align*}
\left|\nabla^{4-k} \Psi(x) \nabla^k H(t,x) \right| &\leq C_\omega \left|v\right|^k  e^{-\delta \sqrt{\omega} \left| v \right| t}, \text{ where } k=1,2,3. \\
 \left| \nabla^{2-k} \left( \Psi(x) ( 1-\Psi^2(x)) \right) \nabla^k (\left|H(t,x)  \right |^2 H(t,x) ) \right| &\leq C_\omega \left|v\right|^k  e^{-\delta \sqrt{\omega} \left| v \right| t}, \text{ where } k=1,2.
\end{align*}
 Where, \begin{align*}
    \nabla^3 f   \; \nabla^1 g &= \sum_{k=1}^{3} \left(  \partial_{x_i {x_j} x_k} f \right)  \;  \left( \partial_{x_k} g \right), \qquad \quad 
     \nabla^2 f \; \nabla^2 g=  \sum_{k=1}^{3} \left(\partial_{x_j x_k} f \right) \left( \;  \partial_{x_i x_k} g \right) , \\
    (\nabla^2 f) \;  g&=  \left( \partial_{x_i x_j} f  \right) g \; , \qquad \quad \qquad \qquad 
    \nabla^1 f \; \nabla^1 g =   \left( \partial_{x_i} f  \right) \left( \partial_{x_j} g \right) . 
 \end{align*}
\end{clm}
 \begin{proof}
  We postpone the proof of Claim \ref{claim calcule} to Appendix \ref{Appendix du theorem Grand vitesse}
 \end{proof}
 
 By the Claim \ref{claim calcule}, we have 
\begin{equation*}
    \left\| \nabla^2 A_0(t) \right\|_{L^2} \leq C_\omega \left| v\right|^3  e^{-\delta \sqrt{\omega} 
    \left| v \right| t }.
\end{equation*}
This concludes the proof of \eqref{EquationDeA_0}.\\ 

Thus, we obtain \begin{equation*}
        \left\| \rm{J}_0(t) \right\|_{H^2} \leq \int_t^{+\infty} \left\|A_0(\tau)  \right\|_{H^2} d\tau  
        \leq C_\omega \left|v\right|^2  e^{-\delta \sqrt{\omega} \left| v \right| t } ,
    \end{equation*}
    \begin{equation*}
        \left\| \rm{J}_0 \right\|_{E} \leq \frac{C_\omega}{  \left| v \right|} .
    \end{equation*}
    
   \item \textbf{Estimate for $\rm{J}_1$.} \\
   Recall that $A_1(r_\omega(t,x))=-R(t,x)\overline{r}_\omega(t,x)-2 \left|R(t,x)\right|^2 r_\omega(t,x) .$ \\ 
   
   Using the elementary interpolation inequality \eqref{interpolation-inequality}, we have
  \begin{align*}
  \left\| \rm{J}_1(r_\omega(t) )\right\|_{H^2}  &\leq  \int_t^{+\infty} \left\| A_1(r_\omega(\tau)) \right\|_{H^2} d\tau \\ &\leq C \int_t^{+\infty} \left\| A_1(r_\omega(\tau)) \right\|_{L^2} d\tau  +  C \int_t^{+\infty} \left\| \nabla^2 A_1(r_\omega(\tau)) \right\|_{L^2} d\tau.
  \end{align*} 
  Let us prove that there exists $C_\omega>0$  such that
  \label{LemmaestimationJ_1}
   \begin{align}
   \label{A_1-L^2}
    \displaystyle \int_t^{+\infty} \left\| A_1(r_\omega(\tau)) \right\|_{L^2} d\tau &\leq \frac{C_\omega}{ \left| v \right|} e^{-\delta \sqrt{\omega} \left| v \right| t} \left\| r_\omega \right\|_{E},\\   
    \label{nabla A_1-L^2}
   \displaystyle \int_{t}^{+\infty} \left\| \nabla^2 A_1( r_\omega(\tau)) \right\|_{L^2} d\tau &\leq C_{\omega} \left| v \right|^2 e^{-\delta \sqrt{\omega} \left|v\right| t } \left\| r_\omega \right\|_{E} .
   \end{align}
\begin{align*}
    \int_{t}^{+\infty} \left\| A_1(r_\omega(\tau)) \right\|_{L^2} d\tau \leq  C \int_{t}^{+\infty} \left\| r_\omega(\tau) \right\|_{L^2} d\tau &\leq C \int_{t}^{+\infty} e^{-\delta \sqrt{\omega} \left| v \right| \tau } d\tau \left\| r_\omega\right\|_{E} \\& \leq \frac{C_\omega}{  \left| v \right|} e^{-\delta \sqrt{\omega} \left|v \right| t  } \left\| r_\omega\right\|_{E} .
\end{align*}
This prove the first estimate. Now, let us look to the second estimate 
\begin{align*}
    \int_{t}^{+\infty} \left\| \nabla^2 A_1(r_\omega(\tau))\right\|_{L^2}  d\tau & \leq \; C  \underbrace{\int_{t}^{+\infty} \left\| \nabla^2 R^2(\tau) \right\|_{L^{\infty}} \left\| r_\omega(\tau) \right\|_{L^2} d\tau }_{\rm I_1}\\ & + C  \underbrace{\int_{t}^{+\infty} \left\|  \nabla R^2(\tau) \right\|_{L^{\infty}} \left\|  \nabla r_\omega(\tau) \right\|_{L^2} d\tau }_{\rm I_2} \\ & + C \underbrace{\int_{t}^{+\infty} \left\| R^2(\tau) \right\|_{L^{\infty}} \left\| \nabla^2 r_\omega(\tau) \right\|_{L^2} d\tau }_{\rm I_3}.
 \end{align*}
 It is easy to see that 
 \begin{align*}
 \left| \rm{I}_1 \right| & \leq C_\omega \left| v \right| e^{-\delta \sqrt{\omega } \left| v \right| t} \left\|r_\omega \right\|_{E}, \\ 
\left| \rm{I}_3\right| & \leq C_\omega \left| v \right|^2 e^{-\delta \sqrt{\omega } \left| v \right| t} \left\|r_\omega \right\|_{E} .
\end{align*}
For $ \rm{I}_2 $ we use the elementary interpolation inequality \eqref{interpolation-inequality},
\begin{equation*}
    \left\| \nabla r_\omega(\tau) \right\|_{L^2} \leq  \left\| \nabla^2 r_\omega(\tau) \right\|_{L^2}^{\frac{1}{2}} \left\| r_\omega \right\|_{L^2}^{\frac{1}{2}}.
\end{equation*}
Thus we get, $$ \left| \rm{I}_2 \right| \leq C_\omega \left| v\right|^{\frac{3}{2}} e^{-\delta \sqrt{\omega} \left| v \right| t } \left\| r_\omega(\tau) \right\|_E .$$
And this concludes the proof of the estimates \eqref{A_1-L^2} and \eqref{nabla A_1-L^2}.  \\

Due to \eqref{A_1-L^2}, \eqref{nabla A_1-L^2} and the fact that $\left|v \right|>1$ we have 
$$ \left\| \rm{J}_1(r_\omega) \right\|_{H^2} \leq C_\omega  \left| v \right|^2   e^{-\delta \sqrt{\omega} \left| v \right| t } \left\| r_\omega \right\|_E .$$

 Then
    \begin{equation*}
        \left\| \rm{J}_1(r_{\omega}) \right\|_{E} \leq \frac{C_\omega}{ \left| v \right|  } \left\|r_\omega \right\|_{E}.
    \end{equation*}
    \item \textbf{Estimate for $\rm{J}_2$.} \\
    Recall that $A_2(r_\omega(t,x))= -\overline{R}(t,x) r_\omega^2(t,x) -2 R(t,x)\left| r_\omega(t,x) \right|^2 .$ 
    $$ \left\| \rm{J}_2(r_\omega(t) \right\|_{H^2} \leq C \int_{t}^{+\infty} \left\| A_2(r_\omega(\tau) \right\|_{H^2} d\tau .  $$
    Using the fact that $H^2$ is an algebra we obtain \begin{align*}
        \left\|  \rm{J}_2(r_\omega(t) \right\|_{H^2} &\leq C \int_{t}^{+\infty} \left\|R(\tau) \right\|_{H^2} \left\| r_\omega(\tau) \right\|_{H^2}^2 d\tau \\ &\leq  C_\omega \left| v \right|^2 \int_{t}^{+\infty}  \left| v \right|^6 e^{-2 \delta \sqrt{\omega} \left| v \right| \tau } d\tau \left\| r_\omega \right\|_{E}^2 \\ &\leq  C_\omega \left| v \right|^7 e^{-2 \delta \sqrt{\omega} \left| v \right| t }  \left\| r_\omega \right\|_{E}^2 .
    \end{align*}
 Then $$\left\| \rm{J}_2(r_\omega) \right\|_E \leq  C_{\omega}  \left| v \right|^4 e^{-\delta \sqrt{\omega } \left| v \right| T_0} \left\| r_\omega \right\|_{E}^2 .$$
 \item \textbf{Estimate for $\rm{J}_3$.} \\
    We have $A_3(r_\omega(t,x))= -\left| r_\omega(t,x) \right|^2 r_\omega(t,x). $
    \begin{align*}
        \left\|\rm{J}_3(r_\omega(t)) \right\|_{H^2} &\leq \int_{t}^{+ \infty} \left\| A_3(r_\omega(\tau)) \right\|_{H^2}^3 d\tau \leq \int_{t}^{+\infty} \left\| r_\omega(\tau) \right\|_{H^2}^3 d\tau  \\ &\leq \int_{t}^{+\infty} \left| v \right|^9 e^{-3 \delta \sqrt{\omega} \left| v \right| \tau} d\tau \left\|r_\omega \right\|_{E}^3 \\ &\leq {C_\omega}\left| v \right|^8 e^{-3 \delta \sqrt{\omega} \left| v \right| t} \left\| r_\omega \right\|_{E}^3
    \end{align*}
   This implies that \begin{equation*}
       \left\| \rm{J}_3(r_\omega) \right\|_{E} \leq C_{\omega} \left| v \right|^5 e^{-2 \delta \sqrt{\omega} \left| v \right| T_0} \left\| r_\omega \right\|_E^3.
   \end{equation*}

   \item \textbf{Stability of $\Phi$.} \\
Recall that $ \displaystyle \Phi(r_\omega(t,x))= -i \, \rm{J}_0(t) - i \sum_{k=1}^{3} \rm{J}_k(r_\omega(t,x)) $. \\ 

    Using the fact that the velocity $v$ is large enough in each estimate   \eqref{estimation de J_0},\eqref{EstimationJ_1},\eqref{EstimationJ_2} and \eqref{EstimationJ_3}, we get \begin{equation*}
        \forall r_\omega \in B_E(0,1), \quad
   \left\| \Phi(r_\omega) \right\|_{E} \leq \left\| \rm{J}_0 \right\|_{E} +  \sum_{k=1}^{3} \left\| \rm{J}_k(r_\omega) \right\|_E \leq 1.
      \end{equation*}
\end{enumerate}
\end{proof}
\item Step 2 : Contraction mapping.  \\
\label{Step2}
Let $f,g \in B_E(0,1)$
\begin{align*}
    \left\| \Phi( f(t) ) - \Phi(g(t)) \right\|_{H^2} \leq \bigg\| \underbrace{\int_{t}^{+\infty} S(t-\tau) \left( A_1(f(\tau)) -A_1(g(\tau))\right)  d\tau }_{\rm{J}_1(f)-\rm{J}_1(g)}\bigg\|_{H^2} \\ + \bigg\|\underbrace{\int_{t}^{+\infty} S(t-\tau) \left( A_2(f(\tau)) -A_2(g(\tau))\right)  d\tau }_{\rm{J}_2(f)-\rm{J}_2(g)}\bigg\|_{H^2}\\ + \bigg\| \underbrace{\int_{t}^{+\infty} S(t-\tau) \left( A_3(f(\tau)) -A_3(g(\tau))\right) d\tau  }_{\rm{J}_3(f)-\rm{J}_3(g)} \bigg\|_{H^2}.
\end{align*}
\begin{lem}
For all $T_0>0,\omega>0$, there exists $V_0>0$ such that for $|v|>V_0$, for all $f,g\in B_E$, we have 
\begin{equation*}
   d_{E}( \Phi(f)-\Phi(g) )  \leq \frac{1}{2}\;  d_{E}(f- g).
\end{equation*}
\end{lem}
\begin{proof}
Due to Lemma \ref{lemmaEstimationOn_J_K} we have 
    $$ \left\|\rm{J}_1(f)-\rm{J}_1(g) \right\|_{E} \leq \frac{C_\omega}{ \left|v\right| } \left\| f- g\right\|_{E}, $$
    Let $V_0>0$ large enough to be chosen below such that for $\left|v\right| > V_0$, we have \begin{equation}
         \label{1-ConditionOnv}
         \frac{1}{\left|v\right|} \leq \frac{1}{8}.
         \end{equation} Then, \begin{equation}
        \label{estimationONJ_1 avec Vitesse}
        \left\| \rm{J}_1(f) - \rm{J}_1(g) \right\|_{E} \leq \frac{1}{8} \;  \left\| f-g \right\|_{E}.
    \end{equation}
    Recall that $A_2(h(t,x))= -\overline{R}(t,x)) h^2(t,x) -2 R(t,x) \left| h(t,x) \right|^2 . $
    \begin{align*}
        \left\|\rm{J}_2(f(t))-\rm{J}_2(g(t)) \right\|_{H^2} &\leq C \int_{t}^{+\infty} \left\|R  (\tau) \right\|_{H^2} \left\| f^2(t) -g^2(t) \right\|_{H^2} d\tau \\ 
       & \leq C_{\omega} \left|v\right|^2 \int_{t}^{+\infty} \left|v\right|^6 e^{-2\delta \sqrt{\omega } \left|v \right| \tau } d\tau \left( \left\|f \right\|_E +\left\| g\right\|_E \right) \left\|f-g  \right\|_E \\ & \leq  C_{\omega}  \left|v\right|^7 e^{-\delta \sqrt{\omega } \left|v \right| T_0 } e^{-\delta \sqrt{\omega } \left|v \right| t }  \left( \left\|f \right\|_E +\left\| g\right\|_E \right) \left\|f-g  \right\|_E .
    \end{align*}
   This implies that $$\left\|\rm{J}_2(f(t))-\rm{J}_2(g(t)) \right\|_{E} \leq C_{\omega}  \left|v\right|^4 e^{-\delta \sqrt{\omega} \left| v \right| T_0}\left( \left\|f \right\|_E +\left\| g\right\|_E \right) \left\|f-g  \right\|_E . $$
   Since the velocity $v$ is large enough we have
   \begin{equation}
    \label{2-ConditionOnv}
     \forall f,g \in B_{E}(0,1), \; \;  C_{ \omega} \left|v\right|^4 e^{-\delta \sqrt{ \omega } \left| v \right| T_0} \left( \left\|f \right\|_E +\left\| g\right\|_E \right) \leq \frac{1}{8} , 
     \end{equation}
    then 
    \begin{equation}
    \label{EstimationJ_2 Vitesse}
    \left\| \rm{J}_2(f)-\rm{J}_3(g) \right\|_E \leq \frac{1}{8 }\left\| f-g\right\|_E .
    \end{equation}

Recall that, $A_3(h(t,x))= -h(t,x) \left| h(t,x) \right|^2.$
\begin{align*}
    \left\| \rm{J}_3(f(t))-\rm{J}_3(g(t)) \right\|_{H^2} &\leq   \int_{t}^{+\infty} \left\| \left| f(\tau)\right|^2 f(\tau)  - \left| g(\tau) \right|^2 g(\tau) \right\|_{H^2} d\tau  \\ &\leq \int_{t}^{+\infty} 
    \left\|  \bar{f}(\tau)    \left( f^2(\tau) -g^2(\tau)\right)  + g^2(\tau) \left( \bar{f}(\tau) - \bar{g}(\tau) \right) \right\|_{H^2}d\tau
    \\ &\leq C \int_{t}^{+\infty}  \left\|f(\tau) - g(\tau) \right\|_{H^2}  \left( \left\|f(\tau) \right\|^2_{H^2} + \left\| g(\tau) \right\|_{H^2}^2 \right) d\tau \\ &\leq 
  C  \int_{t}^{+\infty} \left|v\right|^9 e^{-3 \delta \sqrt{\omega}  \left| v \right| \tau } d\tau  \left( \left\|f\right\|^2_{E} + \left\| g \right\|_{E}^2 \right) \left\|f - g \right\|_{E}   \\
 & \leq  C_{\omega} \left|v\right|^8 e^{-2 \delta \sqrt{\omega}  \left| v \right|T_0 } e^{- \delta \sqrt{\omega}  \left| v \right| t }  \left( \left\|f \right\|^2_{E} + \left\| g \right\|_{E}^2 \right) \left\|f - g \right\|_{E} .
 \end{align*}
 Hence $$  \left\| \rm{J}_3(f(t))-\rm{J}_3(g(t)) \right\|_{E} \leq C_\omega \left| v \right|^5  e^{-2 \delta \sqrt{\omega}  \left| v \right|T_0 }   \left( \left\|f \right\|^2_{E} + \left\| g \right\|_{E}^2 \right) \left\|f - g \right\|_{E}  .   $$
 Due to the choice of the high velocity $v$ we have 
 \begin{equation}
 \label{3-ConditionOnv}
 \forall f,g \in B_E(0,1), \; \, C_\omega \left| v \right|^5  e^{-2 \delta \sqrt{\omega}  \left| v \right|T_0 }   \left( \left\|f \right\|^2_{E} + \left\| g \right\|_{E}^2 \right) \leq \frac{1}{8},\end{equation}
and thus 
 \begin{equation}
     \label{Estimation_J_3 on Vitesse grand}
     \left\| \rm{J}_3(f) -\rm{J}_3(g) \right\|_{E} \leq \frac{1}{8} \left\| f-g \right\|_{E}. 
 \end{equation}
 The inequalities \eqref{1-ConditionOnv}, \eqref{2-ConditionOnv}, 
 \eqref{3-ConditionOnv} specifies how large $V_0$ needs to be taken and from \eqref{estimationONJ_1 avec Vitesse}, \eqref{EstimationJ_2 Vitesse} and \eqref{Estimation_J_3 on Vitesse grand} we have \begin{equation*}
\forall f,g \in B_{E}(0, 1), \; 
\left\| \Phi(f) - \Phi(g) \right\|_E \leq \frac{1}{2}
\left\| f-g \right\|_E.
 \end{equation*}
Thus $\Phi$ is a contraction mapping for $v$ large enough.

\end{proof}
\item Step 3: Conclusion. \\ 
Due to steps 1 and 2, $\Phi$ is a contraction mapping for high velocity on the complete metric space~$(B_E,d_E)$. 
By the fixed point Theorem there exists a unique solution, $$ r_\omega(t,x)= \Phi(r_\omega(t,x))= -i \, \rm{J}_0(t) - i \sum_{k=1}^{3} \rm{J}_k(r_\omega(t,x)) , $$
such that $$ \left\| r_\omega(t) \right\|_{H^2} \leq C_{\omega} \left| v \right|^3 e^{ - \delta \sqrt{\omega} \left| v \right| t} \quad \forall t \in [T_0,+\infty) , $$
which concludes the proof of Theorem
\ref{theoremWithVelocityGrand}.
\end{itemize}
\end{proof}

\end{section}

\appendix
\section{Proof of some Technical results}
\label{Appendix du theorem pranp}

\begin{proof}[Proof of Lemma \ref{spectalprop}]
Recall that for all $f \in H^1\backslash \{ \lambda Q \, ; \; \lambda \in \R \}$ real valued, we have $\int (L^{-}f)f>~0$.
Denote $y_1=\re(\mathcal{Y}^{+})$ and  $y_2=\im({\mathcal{Y}^{+}})$. Since $y_2$ is not colinear to $Q$, we have \begin{equation} \label{normalize}
-\im \int \mathcal{Y}^{+} \overline{\mathcal{Y}^{-}} = 2 \int y_1 y_2 = \frac{2}{e_0} \int -(L^{-}y_2 ) y_2 \neq 0.
\end{equation}

Let $h \in H^1$ such that $h=h_1+ih_2$, we can write $h$ as,
$$h = h^{\perp}+ g ,$$ where,  \begin{equation*}\begin{cases}
h^{\perp} \in G^{\perp}=\{ h \in H^1, \; (h,iQ)=(h,i\mathcal{Y}^{\pm}) =(h,\partial_{x_j} Q)=0 , \; j= 1,2,3   \},  \\ 
g \in Span\{i \mathcal{Y^{+}},i\mathcal{Y^{-}},(\partial_{x_j}Q)_{j=1,2,3} , i \,Q   \} . 
    
    \end{cases}
\end{equation*}
Denote by: \begin{equation}
    \begin{cases}
    \phi_1= \mathcal{Y^{+}}, \qquad \qquad \mu_1 = i \mathcal{Y^{-}}. \\ \phi_2= \mathcal{Y^{-}}, \qquad \qquad \mu_2 =   i \mathcal{Y^{+}}.\\
    \phi_k= \partial_{x_{k-2} }Q ,\quad \quad \mu_k= \partial_{x_{k-2} }Q \quad k=3,4,5 .\\ 
    \phi_6= i \, Q ,  \qquad \qquad \mu_6= iQ - \mu_1(\phi_1,iQ) - \mu_2(\phi_2,iQ) . 
    \end{cases}
\end{equation}
Next, one can verify that: $(\phi_j,\mu_k)=\zeta_{j}\,  \delta^k_j$, by \eqref{normalize} we have $\zeta_1,\zeta_2 \neq 0$ and it is clear that  $\zeta_j \neq 0, \; \forall j \in  [\![3;6]\!] $. This implies that $(\phi_j,\mu_j)_j$ is a biorthogonal family then we can write $g$ as the following \begin{align*}
g= \sum_{j=1}^{6} \frac{(h,\mu_j)}{\zeta_j} \phi_j 
&=\frac{1}{\zeta_1}(h,i \, \mathcal{Y^{-}}) \; \mathcal{Y^{+}} +\frac{1}{\zeta_2} (h,i \, \mathcal{Y^{+}}) \; \mathcal{Y^-} +\displaystyle\sum_{j=1}^{3}\frac{1}{\zeta_{j+2}}(h,\partial_{x_j}Q) \, \partial_{x_j}Q \\ &+ \frac{1}{\zeta_6} \left( (h,i Q) - (h , i\mathcal{Y^{-}})(\mathcal{Y^{+}} ,iQ )  - (h, i \mathcal{Y^{+}}) (\mathcal{Y^{+}},iQ )\, \right) i\, Q .
\end{align*} 

We refer to \cite[Proposition 2.7]{DuRo10} for the following coercivity property of $\mathcal{L}$. There exists a constant $c > 0$ such that 
\begin{equation*}
\forall h \in G^{\perp},  \qquad  \Phi(h) \geq c \left\| h \right\|_{H^1}^2,
\end{equation*}
where, $\Phi(h)= \frac{1}{2} (L^{+} h_1,h_1) + \frac{1}{2}(L^{-}h_2,h_2)$. Next, we have \begin{align*}
\left\| h \right\|_{H^1}^2 = \| h^{\perp}+g \|_{H^1}^2 &\leq c \| h^{\perp} \|_{H^1}^2 + c \left\| g \right\|_{H^1}^2  \\  &\leq C \,  \Phi(h) + C \, \left( \im \int  \mathcal{Y^{+}} \, \overline{h}  \right)^2 + C \, \left( \im \int \mathcal{Y^{-}} \, \overline{h}  \right)^2+ C \left( \int Q \, h_2   \right)^2 \\ &+ C  \sum_{j=1}^{3} \left(  \int   \partial_{x_j}Q \, h_1  \right)^2   .
\end{align*}
\end{proof}
\begin{proof}[Proof of Lemma \ref{wellposed}]
We will only prove the local existence statement. The construction of a maximal solution is standard and we omit it. 
Let us recall that the usual Strichartz estimates are also available outside a convex obstacle, see \cite{killip2015riesz} and \cite{Ivanovici10}:
\begin{theorem}
\label{strichartz}
Let $d \geq 2, \; \Omega \subset \R^d  $ be the exterior of a smooth compact strictly convex obstacle. Let $q,\tilde{q}>2$ and $2 \leq r,\tilde{r} \leq \infty$ satisfy the scaling conditions: $\frac{2}{q}+ \frac{d}{r} = \frac{d}{2}=\frac{2}{\tilde{q}}+ \frac{d}{\tilde{r}}$ \\
Then
\begin{equation}
\left\| e^{i t \Delta}u_0 \pm i \, \int_0^{t} e^{i(t-s)\Delta} F(s) ds \right\|_{L_t ^{q} L_x^{r} }   \leq C_s \left( \left\|u_0\right\|_{L^2(\Omega)} + \left\|F \right\|_{L^{{\tilde{q}'}}_t  L^{\tilde{r}^{'}}_x }  \right).
\end{equation}
\end{theorem}
For the proof of the Lemma \ref{wellposed}, we claim the following result . 
\begin{clm}[H\"older's inequalities]
\label{estimate |u-v|}
choose $a$ such that  $(a,p+1)$ be admissible pairs, \\i.e $\frac{2}{a}+\frac{3}{p+1}=\frac{3}{2}$.  \\ 
Let  $\;u,v \in L^{\infty}L^{p+1} \cap L^a L^{p+1} $ \\ 
Then,
\begin{align}
\label{|u|u}
\left\| \left|u \right|^{p-1} u   \right\|_{L^a L^{\frac{p+1}{p}}} &\leq C \left\| u \right\|^{p-1}_{L^{\infty} L^{p+1} } \left\|u \right\|_{L^a L^{p+1}} \\
\label{|u|u-|v|v}
\left\| \left|u\right|^{p-1}u - \left|v\right|^{p-1} v \right\|_{L^a L^{\frac{p+1}{p}}} &\leq C \left( \left\|u \right\|^{p-1}_{L^{\infty} L^{p+1} }+ \left\| v \right\|^{p-1}_{L^{\infty} L^{p+1} }  \right) \left\|u-v\right\|_{{L^{a} L^{p+1} }} \\
\label{|u|u_H^s}
if \; u \in L^{a} W^{s,p+1}, \qquad \left\| \left|u \right|^{p-1} u  \right\|_{L^a W^{s,\frac{p+1}{p}}} &\leq C \left\| u \right\|^{p-1}_{L^{\infty}L^{p+1}} \left\|u \right\|_{L^a W^{s,p+1}}
\end{align}
\end{clm}
\begin{proof}
Note that $a>2$ since $p<5$. 
For the first estimate it suffices to take $v=0$ in the second estimate. So, let us prove the second estimate.\\

We use the following elementary inequality
\begin{equation}
 \forall(\xi,\zeta) \in \C^2, \qquad \left| \left| \xi \right|^{p-1} \xi -  \left| \zeta \right|^{p-1} \zeta  \right| \leq C_p \, \left(\left| \xi \right|^{p-1}+ \left| \xi \right|^{p-1}  \right) \left| \xi - \zeta \right|
\end{equation}     
As a consequence, fixing $t$, we deduce
\begin{align*}
\left\| \left| u \right|^{p-1} u - \left|v \right|^{p-1}v      \right\|_{L^{\frac{p+1}{p}}}  &\leq C_p \left\| \left| u-v \right| \left( \left|u\right|^{p-1} + \left|v \right|^{p-1} \right) \right\|_{L^{\frac{p+1}{p}}} \\ & \leq C_p \left\| u-v \right\|_{L^{p+1} } \left\| \left| u\right|^{p-1} + \left| v \right|^{p-1}      \right\|_{L^{\frac{p+1}{p-1}}} \\ & \leq  C_p \left\| u-v \right\|_{L^{p+1}} \left( \left\| u \right\|_{L^{p+1}}^{p-1} + \left\| v  \right\|_{L^{p+1}}^{p-1} \right).
\end{align*}                                

Then taking the $L^a$-norm in time, we obtain  $\eqref{|u|u-|v|v}.$ \\
Next, we will prove the last estimate $\eqref{|u|u_H^s}$. For that, we have to use some fractional estimate for the non-linearity $\left|u\right|^{p-1}u$. We refer to \cite{killip2015riesz}, for the following Proposition.
\begin{Prop}{(Fractional chain rule)}
\label{Prop-fractional-rule}\\
Suppose $G \in C^1(\C),s\in(0,1], \,and  \;1<p,p_1,p_2<\infty$ are such that $\frac{1}{p}= \frac{1}{p_1}+\frac{1}{p_2}$ \\  and $0<s<\min{(1+\frac{1}{p_2},\frac{d}{p_2})}.$ Then
\begin{equation}
\label{functional chain rule}
    \left\|(-\Delta_{\Omega})^{\frac{s}{2}}G(f))\right\|_{L^p(\Omega)} \leq \left\|G'(f) \right\|_{L^{p_1}(\Omega)} \left\| (-\Delta_{\Omega})^{\frac{s}{2}} f \right\|_{L^{p_2}(\Omega)},
\end{equation}
Uniformly for $f\in C_{c}^{\infty}(\Omega).$
\end{Prop}   
\begin{rem}
For the sake of simplicity, we will write the Dirichlet Laplacian as $\Delta$ instead of~$\Delta_{\Omega}. $
\end{rem}
By $\eqref{functional chain rule}$, we have 
\begin{equation}
\label{-Delta-u}
\left\|  (-\Delta)^{\frac{s}{2}} \left| u \right|^{p-1} u \right\|_{L^{\frac{p+1}{p}}} \leq C \left\| u  \right\|_{L^{p+1}}^{p-1} \left\| (-\Delta)^{\frac{s}{2}} u \right\|_{L^{p+1}} .
\end{equation}
Due to $\eqref{|u|u}$ and the above estimate $\eqref{-Delta-u}$, we have \begin{align*} 
\left\| \left|u \right|^{p-1} u  \right\|_{L^a W^{s,\frac{p+1}{p}}} &= \left\| (1-\Delta)^{\frac{s}{2}} \left| u \right|^{p-1} u  \right\|_{L^a L^\frac{p+1}{p}} \\  &\leq  C \left\| \left| u\right|^{p-1} u \right\|_{L^a L^\frac{p+1}{p}} +  C \left\| (-\Delta)^{\frac{s}{2}} \left|u\right|^{p-1} u \right\|_{ L^a L^\frac{p+1}{p}}  \\  &\leq C \left\| u \right\|^{p-1}_{L^{\infty}L^{p+1}}  \left\|u\right\|_{L^a L^{p+1}} +      \left\| u \right\|^{p-1}_{L^{\infty}L^{p+1}} \left\|u \right\|_{L^a W^{s,p+1}} \\ &\leq C \left\| u \right\|_{L^{\infty} L^{p+1}}^{p-1} \left\|u \right\|_{L^a W^{s,p+1}} .
\end{align*}
\end{proof}
Fix $M>0$ to be specified later. 
Let $B$ be the ball of $X=C([-T,T],L^2) \cap L^{\infty}_t H^s_x \cap L^a_t W^{s,p+1}_x  $, with radius $M>0$ and  center $0$, i.e the set of functions $u \in X$ such that  $$\left\|u \right\|_{L^{\infty}H^s} \leq M   \text{ and }   \left\|u \right\|_{L^a W^{s,p+1} } \leq M .$$ 
Denote 
$$ d_{B}(u,v)= \left\|u-v  \right\|_{L^{\infty}L^2} + \left\| u-v\right\|_{L^a L^{p+1}} $$
\begin{lem}
$(B,d_B)$ is a complete metric space. 
\end{lem}
\begin{proof}
It is an immediate consequence of the easy fact that $B$ is a closed subset of the following Banach space $$ \displaystyle Y:=C([-T,T],L^2)\cap L^a_t L^{p+1}_x.$$
\end{proof}
For $v \in B $ we define $\Phi(v)(t):= e^{it\Delta } u_0 + D(v)(t) ,$ where $D(v)$ is the Duhamel term given~by $$  D(v)(t):= - \,i \,\displaystyle \int_{0}^{t} e^{i(t-s)\Delta} \left|v(s) \right|^{p-1} v(s) ds  .  $$ 

\begin{itemize}
\item Step 1 : Stability of $B$.  \\
We will prove that: for $v \in B \Longrightarrow \Phi(v) \in B ,$ for a good choose of $M$ and $T$. \\
We have $$\left\|e^{it \Delta } u_0 \right\|_{L^{\infty}H^s }= \left\|u_0 \right\|_{H^s}   \leq \frac{M}{2}  .$$
If the following conditions satisfied \begin{equation}
\label{estimate_h^s}
    2 \left\|u_0 \right\|_{ H^s} \leq M ,
\end{equation}
and we have $$\left\|e^{it \Delta}u_0  \right\|_{L^a W^{s,p+1}}=\left\|(1-\Delta)^{\frac{s}{2}} e^{it \Delta }u_0 \right\|_{L^a L^{p+1}}= \left\| e^{it \Delta} (1-\Delta)^{\frac{s}{2}} u_0 \right\|_{L^a L^{p+1}}  .$$
Using Strichartz estimate (recall that $(a,p+1)$ is an admissible pair) we obtain \begin{equation*}
\left\|e^{it\Delta}u_0 \right\|_{L^a W^{s,p+1}} \leq C_s \left\|(1-\Delta)^{\frac{s}{2}}u_0 \right\|_{L^2}=C_s \left\|u_0 \right\|_{H^s} \leq \frac{M}{2} .
\end{equation*} 
If $M$ is chosen so that \begin{equation}
\label{estimate_L^a_w^s,p}
    M \geq 2 C_s \left\|u_0 \right\|_{H^s}.
\end{equation}
So, we have proved that \begin{equation}
\label{max-M/2}
    \max \left( \left\|e^{it \Delta } u_0 \right\|_{L^{\infty}H^s },\left\|e^{it\Delta}u_0 \right\|_{L^a W^{s,p+1} }  \right) \leq \frac{M}{2}, \newline \; \text{if \eqref{estimate_h^s},  \eqref{estimate_L^a_w^s,p} are satisfied.}
\end{equation}
We next treat the Duhamel term.
\begin{align*}
\left\|D(v) \right\|_{L^{\infty}H^{s}}&=
\left\| \int_0 ^t e^{i(t-\sigma)\Delta }\left|u(\sigma) \right|^{p-1}u(\sigma) d\sigma  \right\|_{L^{\infty}H^s} \\ &= \left\|\int_0^t  e^{i(t- \sigma)\Delta} (1-\Delta)^{\frac{s}{2}} \left|u(\sigma) \right|^{p-1}u(\sigma) d\sigma  \right\|_{L^{\infty} L^2} .
\end{align*} 
By Strichartz estimate, we have
\begin{equation*}
    \left\|D(v) \right\|_{L^{\infty}H^s} \leq C_s \left\| (1-\Delta)^{\frac{s}{2}} \left|u \right|^{p-1}u \right\|_{L^{a'}L^{\frac{p+1}{1}}}    
   \end{equation*} .

Since $s_p= \frac{3}{2} - \frac{3}{p+1}$ and  $s_p \leq s < 1 $, we have  $H^s(\Omega) \subset L^{p+1}(\Omega)$, see \cite{BuGeTz04a}. \\

Now, using H\"older inequality in the time variable and Claim  \ref{estimate |u-v|} we obtain  
\begin{align*}
    \left\|D(v)\right\|_{L^{\infty} H^s } &\leq C_s \,T^{1-\frac{2}{a}}  \left\| (1-\Delta)^{\frac{s}{2}} \left|u \right|^{p-1} u  \right\|_{L^{a}L^{\frac{p+1}{p}}}   \\ & \leq C \, T^{1-\frac{2}{a}} \left\|u \right\|_{L^\infty L^{p+1}}^{p-1} \left\|u \right\|_{L^a W^{s,p+1}} \\
    & \leq C_{1} \, T^{1-\frac{2}{a}} \left\|u  \right\|^{p-1}_{L^{\infty}H^s} \left\|u \right\|_{L^a W^{s,p+1}} .
\end{align*}
We can obtain the same thing for $L^a W^{s,p+1}$-norm of the Duhamel term.
\begin{align*}
\left\|D(v) \right\|_{L^a W^{s,p+1}}&= \left\| \int_0^t e^{i(t-s) \Delta } (1-\Delta)^{\frac{s}{2}}  \left|u(s) \right|^{p-1} u(s) ds \right\|_{L^a L^{p+1}} \\ &\leq C_s \left\| (1-\Delta)^{\frac{s}{2}} \left| u \right|^{p-1} u   \right\|_{L^{a'}L^{\frac{p+1}{p}}} \\ & \leq C \, T^{1-\frac{2}{a}} \left\|  \left|u \right|^{p-1}u\right\|_{L^a W^{s,p+1} }  \\ &\leq C  \, T^{1-\frac{2}{a}} \left\| u \right\|_{L^{\infty}L^{p+1}}^{p-1} \left\|u \right\|_{L^a W^{s,p+1}} \\ & \leq C_2 \, T^{1-\frac{2}{a}} \left\| u\right\|_{L^{\infty}H^s}^{p-1} \left\|u \right\|_{ L^a W^{s,p+1}}.
\end{align*}
Finally, we have obtained \begin{equation*}
    \left\|D(v) \right\|_{L^{\infty}H^s} + \left\|D(v) \right\|_{L^aW^{s,p+1}} \leq \frac{M}{2} .
    \end{equation*}
  
    If the following condition are satisfied \begin{equation}
      \label{C_1,2}
C_{1,2} T^{1-\frac{2}{a}} M^{p-1} \leq \frac{1}{2}.
    \end{equation}
\item Step 2: Contraction property \\
Let $u,v \in B$,
\begin{align*}
    \left\|\Phi(u)-\Phi(v) \right\|_{L^{\infty}L^2 \cap L^a L^{p+1}}&= \left\|D(u)-D(v) \right\|_{L^{\infty}L^2 \cap L^a L^{p+1}} \\ &\leq C_s \, \left\| \left|u \right|^{p-1} u - \left|v \right|^{p-1}v \right\|_{L^{a'}L^{\frac{p+1}{p}}}  \\ &\leq C_s \, T^{1-\frac{2}{a}} \, \left\| \left|u \right|^{p-1} u - \left|v \right|^{p-1}v \right\|_{L^{a}L^{\frac{p+1}{p}}} \\ &\leq  C \,T^{1-\frac{2}{a}} \left\|u-v  \right\|_{L^a L^{p+1}} \left( \left\|u \right\|_{L^{\infty}L^{p+1}}^{p-1}+ \left\|v \right\|_{L^{\infty}L^{p+1}}^{p-1} \right) \\ &\leq C_3 \, T^{1-\frac{2}{a}} M^{p-1} \left\|u-v \right\|_{L^a L^{p+1}},
\end{align*}
which yields,\begin{equation*}
d_B(\Phi(u)-\Phi(v)) \leq C_3 \, T^{1-\frac{2}{a}} \, M^{p-1}\,d_B(u,v).
\end{equation*}
And this prove that $\Phi$ is a contraction if the following condition is satisfied\begin{equation}
\label{C_3}
 C_3 \, T^{1-\frac{2}{a}} \, M^{p-1} < 1 .
\end{equation}

\item Step 3: Fixed point \\
We have proved that $\Phi$ is a contraction on the metric space $(B,d_B)$ if the conditions $\eqref{max-M/2},\eqref{C_1,2}$ and $\eqref{C_3}$ on $M$ and $T$ hold. However it is easy to find $M$ and $T$ satisfying these conditions, we can take $M $ as  \begin{equation*}
    M:= \max(2,2C_s) \left\| u_0 \right\|_{H^s}.
\end{equation*}                                  
So that $\eqref{max-M/2}$ is satisfied, then $T \leq \tau$, where 
\begin{equation*}
    \tau:= \frac{1}{C \left\|u_0 \right\|_{H^s}^{\frac{(p-1)a}{a-2}} } \; ,
\end{equation*}
and $C:=\max(C_{1,2},C_3)$ is a large constant, so that $\eqref{C_3},\eqref{C_1,2}$ hold for $T \leq \tau.$
 Due to the fixed point theorem there exist a unique solution $u$.
\begin{rem}
Note that we can choose $T$, up to multiplicative constant, as an explicit negative power of $\left\| u_0\right\|_{H^s}.$ Also, if $T$ is chosen as above, then $u \in  L^a_T W^{s,p+1}$ and $$  \left\|u \right\|_{L^a_T W^{s,p+1}} + \left\|u \right\|_{L^{\infty}_TH^s} \leq C \, \left\|u_0 \right\|_{H^s} , $$ for a constant $C>0$ which is independent of $T.$
\end{rem}

It remains to check that $u\in C([-T,T],H^s)$, which will be done in step 4 and in step 5 we prove also the uniqueness of $u$ among the $C([-T,T],H^s)$ solutions. \\

\item Step 4 : Continuity 
 $$u(t)=e^{it \Delta } u_0 + i \int_0^t e^{i(t-s)\Delta} \left|u(s) \right|^{p-1}u(s) ds. $$
It is well known that the function: $t \longmapsto e^{it\Delta} u_0 $ is in $C([-T,T],H^s).$\\
Next, we recall that from Step 1 and Step 2 that $$ t \longmapsto \left|u(t) \right|^{p-1}u(t) \in L^{a'}W^{s,p\frac{p+1}{p}}.  $$\\
By Strichartz inequality, we have that the Duhamel term $D(u)\in C([-T,T],H^s)$.\\ Thus, we get $u=e^{i\Delta}u_0 + D(u) \, \in C([-T,T],H^s).$ \\
\item Step 5: Uniqueness.\\
Let $u$ and $v$ be two solutions in $C([-T,T],H^s(\Omega))$ with the same initial data $u_0$. Then $$u(t) -v(t) = i \, \int_0^t e^{i(t-s)\Delta} \left(\left| u(s)\right|^{p-1}u(s) - \left|v(s) \right|^{p-1}v(s) \right)ds ,  $$
by Strichartz inequality, if $\theta >0$ 
\begin{align*}
\left\|u-v \right\|_{L^{a}_{\theta} L^{p+1}} &\leq C \left\| \left|u \right|^{p-1}u - \left|v \right|^{p-1}v\right\|_{L^{a'}_{ \theta}L^{\frac{p+1}{p}} } \\ & \leq C \left\| u-v\right\|_{L^{a'}_{\theta}L^{p+1}} \left( \left\| u \right\|_{L^{\infty}_{\theta} L^{p+1}}^{p-1}  + \left\| v \right\|_{L^{\infty}_{\theta} L^{p+1}}^{p-1}    \right) \\ & \leq C \, \theta^{1-\frac{2}{a}} \left\|u-v \right\|_{L^a_{\theta} L^{p+1}} \left( \left\|  u\right\|_{L^{\infty}_{\theta}L^{p+1}}^{p-1} + \left\|  v \right\|_{L^{\infty}_{\theta}L^{p+1}}^{p-1} \right)
\end{align*}
Choosing $\theta > 0 $ small enough, so that $$ C \, \left( \left\|  u\right\|_{L^{\infty}_{\theta}L^{p+1}}^{p-1} + \left\|  v\right\|_{L^{\infty}_{\theta}L^{p+1}}^{p-1} \right)  \theta^{1-\frac{2}{a}}  \, < 1 .$$

We deduce that $\left\|u-v \right\|_{L^a_{\theta} L^{p+1}}=0$, then $u=v$ in $[-\theta,\theta].$ Iterating this argument, we obtain $u=v$ in $[-T,T].$ 
\end{itemize} 
\end{proof}
\begin{proof}[Proof of Lemma \ref{modulation}] 
Let $\rho=u-R$ and let
\begin{align*}
\Phi:L^2 &\times  \R^3 \times \R \longrightarrow \R^4 \\
        &(\rho \; ,\; y \; ,\;\mu ) \longrightarrow \left(\re \int (\rho+R-\widetilde{R}) \nabla \widetilde{Q}_{\omega}\Psi e^{- i \left( \frac{1}{2}(x.v) +\theta \right) }\,e^{- i \mu} dx   \; ,\; \im \int (\rho+R-\widetilde{R})  \overline{\widetilde{R}} \,dx \,\right).
\end{align*}

Denote:\begin{align*}
\Phi_1(\rho,y,\mu)&=\displaystyle \re \int (\rho+R-\widetilde{R}) \nabla \widetilde{Q}_{\omega}\Psi e^{- i \left( \frac{1}{2}(x.v) +\theta \right) }\,e^{- i \mu} dx ,\\
\Phi_2(\rho,y,\mu)&= \displaystyle \im \int (\rho+R-\widetilde{R}) \overline{\widetilde{R}}.
\end{align*}
\begin{itemize}
\item Step 1: Compute $d_{(y,\mu)} \Phi_1$. 
Let $z\in \R^3,l \in \R .$ 
\begin{multline}
\label{d_y_Phi_1}
    \left( d_y \Phi_1(\rho,y,\mu) \, . z \right)_j = \re \bigg( z_j\int \partial_{x_j}\widetilde{Q}_{\omega} \partial_{x_j}\widetilde{Q}_{\omega} \Psi^2 \, dx + \sum_{k\neq j} z_k \int \partial_{x_k} \widetilde{Q}_{\omega} \partial_{x_j}\widetilde{Q}_{\omega} \Psi^2 \, dx\\ - \sum_{k=1}^{3} \int  (\rho+R-\widetilde{R}) \partial_{x_k}\partial_{x_j} \widetilde{Q}_{\omega} \Psi e^{- i \left( \frac{1}{2}(x.v) +\theta \right) }\,e^{- i \mu}\, z_k \, dx \bigg).
    \end{multline}
\begin{equation}
\label{d_mu_Phi_1}
\left(d_{\mu}\Phi_1(\rho,y,\mu).l \right)_j = \re \left(i \int l \, (\rho+R-\widetilde{R}) \partial_{x_j}\widetilde{Q}_{\omega} \Psi e^{-i \left( \frac{1}{2}(x.v) +\theta \right)} e^{- i \mu } \, dx \right). 
\end{equation}

\begin{clm}
\label{d_y(x,y,z)}
\begin{align}
\label{est-3}
&\displaystyle \left( d_y \Phi_1(\rho,y,\mu)\, . z\right)_j = z_j \left\|\partial_{x_j} \widetilde{Q}_{\omega} \Psi  \right\|_{L^2}^2+O\left( \left| z \right| \left(  \left\| \rho \right\|_{L^2} + \left|y\right| \right) \right). \\
\label{est-4}
 &d_y \Phi_1(0,0,0) = diag \left( \left\|\partial_{x_j}\widetilde{Q}_{\omega} \Psi \right\|_{L^2}^2 \right). \\
\label{est-6}
 & \displaystyle d_{\mu}\Phi_1(\rho,y,\mu).l\,= O(  \left| l \right| \left( \left\|\rho \right\|_{L^2}+\left|y\right| \right) ). \\
 \label{est-7}
& \displaystyle d_{\mu}\Phi_1(0,0,0)=0 . 
 \end{align}
\end{clm}

\begin{proof}
For the first estimate we have$$ \left| \int \rho \,  \partial_{x_k,x_j} \widetilde{Q}_{\omega} \Psi e^{-i\left( \frac{1}{2}(x.v) +\theta \right)} e^{- i \mu} \right| \leq C \left\| \rho \right\|_{L^2}, $$\\ and \begin{align*}
 \int (R-\widetilde{R} ) \partial_{x_k,x_j} \widetilde{Q}_{\omega} \Psi dx & = \int \int_0^1 \frac{d}{dt} R(x-ty)dt \,  \partial_{x_k,x_j}\widetilde{Q}_{\omega} \Psi dx \\ & = \int \int_0^1 \, y \nabla R(x-ty) dt \, \partial_{x_k,x_j}\widetilde{Q}_{\omega}  \Psi dx , \\ \left|\int (R-\widetilde{R} ) \partial_{x_k,x_j} \widetilde{Q}_{\omega} \Psi dx  \right| & \leq C \left| y \right| .
\end{align*}
This implies that \begin{equation}
 \label{I_1}
\re \left( \sum_{k=1}^{3} \int z_k (\rho+R-\widetilde{R}) \partial_{x_k} \partial_{x_j}\widetilde{Q}_{\omega} \Psi e^{-i \left( \frac{1}{2}(x.v) +\theta \right)} e^{- i \mu }\,  dx\right) = O\left( \left|z\right| (\left\| \rho \right\|_{L^2} + \left| y\right|   ) \right).  \end{equation}
Since $Q_\omega$ is radial, we have   $$\forall k \neq j,  \qquad \int \partial_{x_k}Q  \, \partial_{x_j}Q  \, dx=0 , $$
which yields, for $k \neq j$\begin{align*}
    \int \partial_{x_k}\widetilde{Q}_{\omega} \partial_{x_j}\widetilde{Q}_{\omega} \, \Psi^2 \,  dx \leq \int \partial_{x_k} \, \widetilde{Q}_{\omega} \partial_{x_j}\widetilde{Q}_{\omega}\, dx \; = 0.
\end{align*}
Then \begin{equation}
\label{radial}
\re \,\left( \sum_{k \neq j} z_k \int \partial_{x_k} \widetilde{Q}_{\omega} \partial_{x_j} \widetilde{Q}_{\omega} \Psi^2 \, dx \right) = 0 . 
\end{equation}

The estimate \eqref{est-3} it is a consequence of \eqref{I_1} and \eqref{radial}. Applying \eqref{est-3} at point~$(0,0,0)$, we get \eqref{est-4}.  \\ 

Due to $\eqref{d_mu_Phi_1}$, we have
\begin{equation*}
   \left(d_{\mu}\Phi_1(\rho,y,\mu).l \right)_j =  \re \left( i \int l \, (\rho+R-\widetilde{R}) \partial_{x_j}\widetilde{Q}_{\omega} \Psi e^{-i \left( \frac{1}{2}(x.v) +\theta \right)} e^{- i \mu }  \right).
\end{equation*}    
Then
$$d_{\mu} \Phi_1(\rho,y,\mu).l= \im  \int l \, (\rho+R-\widetilde{R}) \partial_{x_j}\widetilde{Q}_{\omega}  \Psi e^{-i \left( \frac{1}{2}(x.v) +\theta \right)} e^{- i \mu } dx  $$
Similarly to the proof of the estimate \eqref{I_1}, we have 
$$d_{\mu} \Phi_1(\rho,y,\mu).l= O( \left| l \right| \left(   \left\| \rho \right\|_{L^2}+\left| y \right| \right)).$$
Finally, due to the above equality it is easy to see that $$d_{\mu}\Phi_1(0,0,0)=0, $$
which concludes the proof of the Claim \ref{d_y(x,y,z)}
\end{proof}
\item Step 2: Compute $d_{(y,\mu)}\Phi_2$.  \\ 
Recall that $$ \Phi_2(\rho,y,\mu)=\im  \int (\rho+R-\widetilde{R})  \, \overline{\widetilde{R}} .$$
\begin{align}
\label{d_y Phi_2}
      \displaystyle  d_y \Phi_2(\rho,y,\mu).l&= - \im \left(  \sum_{j=1}^3   \int l_j (\rho+R-\widetilde{R}) \partial_{x_j} \widetilde{Q}_{\omega}  \Psi e^{-i\left( \frac{1}{2}(x.v) +\theta \right)} e^{-i \mu ) } \right). \\ 
      \label{d_mu Phi_2}
      \displaystyle d_{\mu}  \Phi_2(\rho,y,\mu) . q &= - \displaystyle \int  q \, \widetilde{Q}_{\omega}^2 \Psi^2 - \re \int  \, q (\rho+R-\widetilde{R}) \; \overline{\widetilde{R}} .
\end{align}
\begin{clm}
\label{d_y_2_Phi}
Let $l \in \R^3, q \in \R .$
\begin{align}
\label{esti--1}
d_y\Phi_2(\rho,y,\mu).l &= O( \left| l \right| \left(   \left\| \rho \right\|_{L^2} + \left| y \right|\right) ) .\\ 
\label{esti--2}
d_y\Phi_2(0,0,0)&= 0 .\\
\label{esti--3}
d_{\mu}\Phi_2(\rho,y,\mu).q &= - \int q \, \widetilde{Q}_{\omega}^2 \Psi^2 + O(\left| q \right| ( \left\| \rho \right\|_{L^2} + \left| y \right| ) ) . \\
\label{esti--4}
d_{\mu}\Phi_2(0,0,0)&= - \left\| \widetilde{Q}_{\omega} \Psi^2\right\|_{L^2}.
\end{align}
\end{clm}
\begin{proof}

Using the same argument as in the proof of Claim \ref{d_y(x,y,z)}, we obtain
\begin{equation*}
 \im \left(\sum_{j=1}^3 \int l_j (\rho+R-\widetilde{R} )  \partial_{x_j} \widetilde{Q}_{\omega}  \Psi e^{-i\left( \frac{1}{2}(x.v) +\theta \right)} e^{-i \mu ) } \right)=O(\left|l \right| ( \left\|\rho \right\|_{L^2} + \left|y\right| ) ) .
\end{equation*}
Due to \eqref{d_y Phi_2}, we obtain the first estimate. Applying \eqref{esti--1} at point $(0,0,0)$, we obtain $$d_y\Phi_2(0,0,0)= 0 . $$
Similarly to the proof of $d_{y,\mu} \phi_1$, we have
$$\re \int  \, q (\rho+R-\widetilde{R}) \overline{\widetilde{R}}=O(\left| q \right| ( \left\| \rho \right\|_{L^2} + \left| y \right| ) ).$$
Using the above estimate and \eqref{d_mu Phi_2}, we get  $$ \displaystyle d_{\mu} \Phi_2(\rho,y,\mu) . q = - \displaystyle \int q \, \widetilde{Q}_{\omega}^2 \Psi^2 +O(\left| q \right| ( \left\| \rho \right\|_{L^2} + \left| y \right| ) ). $$
Then
$$d_{\mu}\Phi_2(0,0,0)= - \left\| \widetilde{Q}_{\omega} \Psi^2\right\|_{L^2}.$$
This concludes the proof of the Claim \ref{d_y_2_Phi}
\end{proof}
\item Step 3: Conclusion  \\ 
From Step 1 and Step 2 we get
\begin{align*}
d_{(y,\mu)}\Phi(\rho,y ,\mu)= \begin{pmatrix} 
\left\| \partial_{x_1} \widetilde{Q}_{\omega} \Psi   \right\|_{L^2}^2 & 0 & 0 & 0 \\ 
0 & \left\| \partial_{x_2} \widetilde{Q}_{\omega} \Psi \right\|_{L^2}^2 & 0 & 0 \\
0 & 0 & \left\| \partial_{x_3} \widetilde{Q}_{\omega} \Psi \right\|_{L^2}^2 & 0 \\
0 & 0 & 0 & -\left\| \widetilde{Q}_{\omega} \Psi \right\|_{L^2}
\end{pmatrix}
+ O(\left\| \rho \right\|_{L^2} + \left| y \right| )  . 
\end{align*}
We can deduce that
$d_{(y,\mu)} \Phi(0,0,0)$ is invertible and we have $\Phi(0,0,0)=0.$\\ 

Then, by the Implicit function theorem, there exists $\varepsilon_0 > 0 ,\;  \varepsilon_0 \leq \eta $ \\ and a $C^1$-function
\begin{align*}
g : B_{L^2}(0,&\varepsilon ) \longrightarrow B_{\R^4}(0,\eta) \\
             &\rho \longmapsto g(\rho)=\left((y(\rho),\mu(\rho) \right)
\end{align*}
such that $\Phi(\rho,y,\mu)=0$ in $B_{L^2}(0,\varepsilon) \times g(B_{L^2}(0;\varepsilon))$ is equivalent to $(y,\mu)=g(\rho)$. 
\\
Finally we set  
\begin{equation*}
r:=r(\rho)= \rho + R - \widetilde{R}(\cdot-y(\rho)) e^{i \mu(\rho)}    .
\end{equation*}
\end{itemize}

\end{proof}

\begin{proof}[Proof of Lemma $\ref{modulatedfinaldata}$]
\begin{align*}
 \sigma : \, &\R^2 \longrightarrow H^1_0 \qquad  \qquad \qquad \qquad \qquad \qquad \Gamma : B_{H^1_0}(\varepsilon) \longrightarrow H^1_0 \times \R^3 \times \R \\ 
         \lambda:=&\lambda^{\pm} \longmapsto i  \left(\lambda^{+}Y_{+}(T_n)+ \lambda^{-} Y_{-}(T_n)  \right) ,  \qquad \qquad \rho \longmapsto (r \,, \,y\,,\, \mu) .
\end{align*}
Where, $(r,y,\mu)$ is the modulation of $u(T_n)$ around $R(T_n)$ and $B_{H^1_0}(\varepsilon)$ is a ball of radius $\varepsilon>0$ which is defined in the proof of the Lemma \ref{modulation}. 
\begin{align*}
    \Lambda:H^{1}_0  \times & \R^3 \times  \R \longrightarrow \R^2 \\
 & (r,y,\mu) \longmapsto  \left( \alpha^{+}(T_n)=\im \int \widetilde{Y}_{-}(T_n,x) \, \overline{r}(T_n,x) dx \, , \, \alpha^{-}(T_n)=\im \int \widetilde{Y}_{+}(T_n,x) \,  \overline{r}(T_n,x) dx \right).
\end{align*}
We have, $\sigma(0)=0$, $\Gamma(0)=(0,0,0)$, $\Lambda(0,0,0)=(0,0) . $ \\  
Denote: $\Theta = \Lambda \circ \Gamma \circ \sigma  .$ \\
Now let us prove that $\Theta$ is a diffeomorphism on a $\mathcal{V}_{0}$ a neighbourhood of $0\in \R^2$ by  \\ computing  $d\Theta=d\Lambda \circ d\Gamma \circ d\sigma$.\\

Firstly, we have that $d \sigma(\lambda)= \sigma $,  for all $\lambda \in \R^2$. Secondly, let $l \in H^1_0,z \in \R^3, q \in\R $ such~that 
\begin{align*}
d \Lambda (r,y,\mu).(l,z,q ) = \bigg( &\im  \int \widetilde{Y}_{-}(x) \,  \overline{l}(x)  - \sum_{j=1}^{3}  z_j  \partial_{x_j} \widetilde{Y}_{-}(x) \, \overline{r}(x)  + i q  \widetilde{Y}_{-}(x) \, \overline{r}(x) \,  dx  \;  , \\  &\im \int \widetilde{Y}_{+}(x) \, \overline{l}(x)  - \sum_{j=1}^{3} z_j   \partial_{x_j} \widetilde{Y}_{+}(x) \, \overline{r}(x) + i q  \widetilde{Y}_{+}(x) \,  \overline{r}(x) \,  dx    \bigg).
\end{align*}
Finally, we have to compute $d \,\Gamma$. Let $\Phi$ and $g$ defined as in the proof of the Lemma \ref{modulation} for~$R(t_n).$ Then, we obtain 
$$ \Gamma(\rho)=\left(\rho+R(T_n)-\widetilde{R}(T_n, \cdot - y(\rho)), y(\rho), \mu(\rho)\right) .$$
\begin{equation}
\label{d_gamma}
d\Gamma(\rho).l = \left(l + \nabla R (T_n,\cdot - y(\rho))e^{i \mu(\rho)} dy(\rho).l + i R(\cdot - y(\rho))e^{i \mu(\rho)} d\mu(\rho).l \, , \,  dy(\rho) \, , \,  d\mu(\rho) \right) .   
\end{equation}
 we have \begin{equation*}\Phi(\rho,y(\rho),\mu(\rho))=0
  \Longrightarrow
    \begin{cases}
    \Phi_1(\rho,y(\rho),\mu(\rho))&=0 \\
    \Phi_2(\rho,y(\rho),\mu(\rho))&=0

    \end{cases}
    \Longrightarrow
    \begin{cases}
    d_1\Phi_1+d_2\Phi_1 \, dy(\rho) + d_3 \Phi_1 \, d\mu(\rho) = 0 \\ 
    d_1 \Phi_2 + d_2 \Phi_2 \,dy(\rho)+ d_3\Phi_2 \, d\mu(\rho) =0
    \end{cases}
\end{equation*}
\begin{equation}
\label{dy(rho)-dmu(rho)}
    \Longrightarrow
    \begin{cases}
    dy(\rho)= \left(d_2 \Phi_1 \right)^{-1} \left[  -(d_1 \Phi_1)-(d_3 \Phi_1) d \mu(\rho)   \right] \\
    d\mu(\rho)=  \left(d_3 \Phi_2 \right)^{-1}(d_2 \Phi_2) \left( d_2 \Phi_1 \right)^{-1}(d_1 \Phi_1)  - \left( d_3 \Phi_2\right)^{-1} (d_1 \Phi_2)
    - \left( d_3 \Phi_1 \right)^{-1} (d_1 \Phi_1) \\ \qquad \quad   + \left( d_3 \Phi_1 \right)^{-1} (d_2 \Phi_1)\left(d_2 \Phi_2 \right)^{-1} (d_1 \Phi_2) .
    \end{cases}
\end{equation}
 Recall that $$d\Theta(\lambda).\tilde{\lambda} = d\Lambda(d \, \Gamma(\sigma(\lambda)))\,. \, d\, \Gamma(\sigma(\lambda)) \,  . \,  \sigma(\tilde{\lambda}) , $$ 
by $\eqref{d_gamma}$ we get
\begin{align*}
d\,\Gamma(\sigma(\lambda)).\sigma(\tilde{\lambda})=&   \bigg(\sigma(\tilde{\lambda})  +\nabla R(T_n,\cdot-y(\sigma(\lambda))e^{i\mu (\sigma(\lambda))} dy({\sigma(\lambda)}).\sigma(\tilde{\lambda}) \\ & + i \,  R(T_n,\cdot-y(\sigma(\lambda))) e^{i \mu(\sigma(\lambda))} d\mu(\sigma(\lambda)).\sigma(\tilde{\lambda})\;  , \;     dy(\sigma(\lambda)).\sigma(\tilde{\lambda}) \; , \;  d\mu(\sigma(\lambda)).\sigma(\tilde{\lambda})    \bigg).
\end{align*}
We claim the following estimate which will be proved at the end of this proof.
\begin{clm}
\label{dy.sigma.rho}
Let $\delta>0$ such that \begin{align*}
    dy(\sigma(\lambda)).\sigma(\tilde{\lambda}) &= O\left( (  e^{-\delta \sqrt{\omega} \left|v\right|T_n }   + \left|\lambda\right| ) |\tilde{\lambda} | \right)   \\
    d\mu(\sigma(\lambda)).\sigma(\tilde{\lambda}) &= O\left( (  e^{-\delta \sqrt{\omega} \left|v\right|T_n }   + \left|\lambda\right| ) |\tilde{\lambda} | \right).
\end{align*}
\end{clm}
By the claim above we have 
\begin{equation*}
    d\Gamma(\sigma(\lambda)).\sigma(\tilde{\lambda}) = \left(\sigma(\tilde{\lambda}) \, , 0 \, , 0 \right)+ O\left( (  e^{-\delta \sqrt{\omega} \left|v\right|T_n }   + \left|\lambda\right| ) |\tilde{\lambda} | \right).
\end{equation*}
Using the expression of $d\Lambda$, we get
\begin{align*}
d \Theta (\lambda).\tilde{\lambda} &= d \Lambda(\sigma(\lambda)).(\sigma(\tilde{\lambda}),0,0)+ O\left( (  e^{-\delta \sqrt{\omega} \left|v\right|T_n }   + \left|\lambda\right| ) |\tilde{\lambda} | \right), \\ 
 d \Theta  (\lambda)  &= \mathcal{M}+  O\left(   e^{-\delta \sqrt{\omega} \left|v\right|T_n }   + \left|\lambda\right|  \right).
\end{align*}
Where $\mathcal{M}$ is a matrix such that\\ \begin{equation*}
\mathcal{M}=\begin{pmatrix}
\displaystyle \re \int \widetilde{Y}_{-}(T_n,x) \overline{Y}_{+}(T_n,x) dx  &\displaystyle \re \int \widetilde{Y}_{-}(T_n,x) \overline{Y}_{-}(T_n,x) dx \\ \displaystyle \re \int \widetilde{Y}_{+}(T_n,x) \overline{Y}_{+}(T_n,x) dx  & \displaystyle \re \int \widetilde{Y}_{+}(T_n,x) \overline{Y}_{-}(T_n,x) dx
\end{pmatrix}
\end{equation*}

Since $ \mathcal{Y_{+}}$ and $ \mathcal{Y_{-}}$ are linearly independent, then the following matrix is invertible \begin{equation*}
\mathcal{A}=\begin{pmatrix}
\displaystyle \re \int {Y}_{-}(T_n,x) \overline{Y}_{+}(T_n,x) dx  &\displaystyle \re \int {Y}_{-}(T_n,x) \overline{Y}_{-}(T_n,x) dx \\ \displaystyle \re \int {Y}_{+}(T_n,x) \overline{Y}_{+}(T_n,x) dx  & \displaystyle \re \int {Y}_{+}(T_n,x) \overline{Y}_{-}(T_n,x) dx
\end{pmatrix}
\end{equation*}
And we have $$ \left| \re \int \left( \widetilde{Y}_{-}(T_n,x) - {Y}_{-}(T_n,x) \right) \overline{Y}_{+}(T_n,x) \, dx \right| \leq C \left|  y \right|  \leq C \left| \lambda \right|.$$
We deduce that $\mathcal{M}$ is invertible, thus  $d\Theta$ is invertible on a some ball $B_{\R^2}(\beta)$. This implies that $\Theta$ is a diffeomorphism from the ball $B_{\R^2}(\beta)$\; ($\beta>0$ independent of $n$ for n large enough) to some neighborhood $\mathcal{U}$ of $0 \in \R^2$.\\

Let $\eta>0$ be such that $B_{\R^2}(\eta) \subset \mathcal{U} $. Then, for any $\alpha^{+} \in B_{\R}(\eta)$, there exist a unique \\  $\lambda=\lambda(\alpha^{+}) \in B_{\R^2}(\beta)$ such that\begin{equation*}
\Theta ( \lambda(\alpha^{+})) = (\alpha^{+} \,, \,0 ) \quad  and  \quad \left|  \lambda(\alpha^{+})  \right| \leq C \left| \alpha^{+} \right| .
\end{equation*}
And this concludes the proof of Lemma \ref{modulatedfinaldata}.
\end{proof}
\begin{proof}[Proof of Claim \ref{dy.sigma.rho}]
From \eqref{dy(rho)-dmu(rho)}, we have 
\begin{align*}
    dy(\rho)&= \left(d_2 \Phi_1 \right)^{-1} \left[  -(d_1 \Phi_1)-(d_3 \Phi_1) d \mu(\rho)   \right] , \\ d\mu(\rho)&=  \left(d_3 \Phi_2 \right)^{-1}(d_2 \Phi_2) \left( d_2 \Phi_1 \right)^{-1}(d_1 \Phi_1)  - \left( d_3 \Phi_2\right)^{-1} (d_1 \Phi_2)
    - \left( d_3 \Phi_1 \right)^{-1} (d_1 \Phi_1) \\  & - \left( d_3 \Phi_1 \right)^{-1} (d_2 \Phi_1)\left(d_2 \Phi_2 \right)^{-1} (d_1 \Phi_2) .
\end{align*}
 Remark that it suffices to prove that  
 \begin{align*}
     d_1 \Phi_1 . \sigma(\tilde{\lambda}) &= O\left(  | \lambda \;  \tilde{\lambda} |  \right) \\ 
     d_1 \Phi_2. \sigma(\tilde{\lambda})&= O\left( (  e^{-\delta \sqrt{\omega} \left|v\right|T_n }   + \left|\lambda\right| ) |\tilde{\lambda}| \right) .
 \end{align*}
Let $l \in H^1_0$, we have 
\begin{align*}
\displaystyle  d_1 \Phi_1 (\rho,y,\mu).l  &= \re \displaystyle\int l(x) \; \nabla \widetilde{Q}_{\omega}(T_n,x) \Psi(x)  e^{- i \widetilde{ \varphi } (T_n,x)} dx ,  \\
 \displaystyle d_1 \Phi_2 (\rho,y,\mu).l &= \im \displaystyle \int l(x) \;  \overline{\widetilde{R}}(T_n,x) dx .
   \end{align*}
Recall that $ \displaystyle \sigma(\tilde{\lambda})= i \left( \tilde{\lambda}^{+} Y_{+}(T_n,x) + \tilde{\lambda}^{-} Y_{-}(T_n,x) \right)$.
\begin{align*}
d_1 \Phi_1. \sigma(\tilde{\lambda}) & = 
\re \int i\left(\tilde{\lambda}^{+} 
Y_{+ } + \tilde{\lambda}^{-} Y_{-} \right)\nabla \widetilde{Q}_{\omega} \Psi  e^{-i \widetilde{ \varphi } } dx \\  &=    \im \bigg[  e^{- i \mu }\tilde{\lambda}^{+}  \underbrace{ \int  \mathcal{Y}_{\omega}^{+} \; \nabla \widetilde{Q}_{\omega} \Psi   dx}_{\rm{I}_1}+  e^{- i \mu } \tilde{\lambda}^{-}  \underbrace{  \int \mathcal{Y}_{\omega}^{-} \;  \nabla \widetilde{Q}_{\omega}  \Psi  dx  }_{\rm{I}_2}\bigg].
\end{align*}
\begin{align*}
 {\rm{I}_1}+{\rm{I}_2}&= \int  \mathcal{Y}_{\omega}^{+} \; \nabla {Q}_{\omega} \Psi   dx + \int \mathcal{Y}_{\omega}^{-} \;  \nabla {Q}_{\omega}  \Psi  dx + O(\left| y \right|). 
\end{align*}

Since $\mathcal{Y}_{\omega}^{\pm}$ and $Q_{\omega}$ are radial, we have
$$  \int  \mathcal{Y}_{\omega}^{\pm} \; \nabla {Q}_{\omega} \Psi \,   dx  \leq  \int  \mathcal{Y}_{\omega}^{\pm} \; \nabla {Q}_{\omega}  \,   dx  =0 , $$

and using  $\left| y \right| \leq \left| \lambda \right|$ we get 
\begin{equation*}
    d_1 \Phi_1.\sigma(\tilde{\lambda})= O\left(  | \lambda  \tilde{\lambda} |  \right).
\end{equation*}
Denote $ \displaystyle y_{1} =\re\left({\mathcal{Y}_{\omega}^{+}}\right)=\re\left({\mathcal{Y}_{\omega}^{-}}\right)$ and $y_{2}= \im\left({\mathcal{Y}_{\omega}^{+}}\right)= - \im \left({\mathcal{Y}_{\omega}^{-}}\right).$ \\ 
Recall that  $\mathcal{L}_{\omega} \mathcal{Y}_\omega^{\pm}= \pm e_{\omega} \mathcal{Y}_\omega^{\pm}$ . 
\begin{align*}
    d_1\Phi_{2}.\sigma(\tilde{\lambda})&= \im \int i\, \left( \tilde{\lambda}^{+} Y_{+} + \tilde{\lambda}^{-} Y_{-} \right) \widetilde{Q}_{\omega} \Psi \,  e^{-i \widetilde{ \varphi } } dx \\&= \re \bigg[ e^{-i\mu} \tilde{\lambda}^{+} \underbrace{\int \mathcal{Y}_{\omega}^{+}\, \widetilde{Q}\,  \Psi  \,  dx }_{\rm{J_1}}+ e^{-i\mu} \tilde{\lambda}^{-}  \underbrace{\int - \mathcal{Y}_{\omega}^{-} \, \widetilde{Q} \,  \Psi\, dx  }_{\rm{J_2}}   \bigg] . 
\end{align*}

\begin{align*}
 {\rm{J}_1}+{\rm{J}_2}&= \int   \left(-{L}^{-}_{\omega}y_2 + i \, {L}^{+}_{\omega}y_1 \right)  \widetilde{Q}_{\omega} \Psi   dx +  \int  - \left({L}^{-}_{\omega}y_2 + i \, {L}^{+}_{\omega}y_1 \right) \,  \widetilde{Q}_{\omega} \Psi   dx \\ &=  - 2\,i \int   {L}^{-}_{\omega}   y_2 \;   (\widetilde{Q}_{\omega} \Psi  ) \,dx.
\end{align*}
Since $L^{-}_{\omega}$ is self-adjoint operator.
\begin{equation*}
       {\rm{J}_1}+{\rm{J}_2} = - 2 \, i \int      y_2 \;  {L}^{-}_{\omega} ( {Q}_{\omega} \Psi  ) \,dx + O(\left| y \right|). 
\end{equation*}
Using the fact that $\partial_{x_j} \Psi$ has a compact support, ${L}^{-}_{\omega} (  {Q}_{\omega} ) = 0$ and $\left| y \right| \leq \left| \lambda \right|$ we get 
\begin{equation*}
    d_1 \Phi_2.\sigma(\tilde{\lambda})= O\left( (e^{-\delta \sqrt{\omega} \left|v\right| T_n} + \left| \lambda \right| ) | \tilde{\lambda} |  \right).
\end{equation*}

This concludes the proof of the Claim \ref{dy.sigma.rho}.

\end{proof}

\section{Computation of some estimates}
\label{Appendix du theorem Grand vitesse}
\begin{proof}[Proof of Claim \ref{claim calcule}] 
Using \eqref{Supp-Compact-Q} and the compact support of 
$\nabla^k \Psi$, we obtain the first estimate. \\ Let us prove the second inequality. \\

Notice that  $F: z\longmapsto \left| z\right|^2 z = z^2 \Bar{z}$ is differentiable on $\C$ and \begin{align*}
\frac{dF}{dz}(z)&= 2 \left|z\right|^2,  \qquad \frac{dF}{(dz)^2}(z)= 2 \Bar{z}, \qquad 
\frac{dF}{d\Bar{z}}(z)= z^2,  \qquad \frac{dF}{(d\Bar{z})^2}(z)= 0 \\
\frac{d^2F}{dzd\Bar{z}}(z)&=\frac{d^2F}{d\Bar{z}dz}(z)=0.
\end{align*}
Since $x\longmapsto H(t,x) $ is smooth. \\ Then we have, 
\begin{align*}
 \nabla \left(\left| H(t,x) \right|^2 H(t,x) \right)&=\nabla H(t,x) \nabla_z F(H(t,x))+ \nabla \overline{H}(t,x) \nabla_{\Bar{z}} F(H(t,x))\\ 
 \nabla^2 \left(\left| H(t,x) \right|^2 H(t,x) \right)  &= \nabla^2 H(t,x) \nabla_z F(H(t,x))+  \nabla^1 H(t,x) \nabla^1 H(t,x) \nabla_{zz} F(H(t,x))\\ &+ \nabla^2 \overline{H}(t,x) \nabla_{\Bar{z}}F(H(t,x)), 
\end{align*} 
where, $ \nabla f= \left( \partial_{x_i} f  \right)_i, \; i=1,2,3. $ \\ 

Using again the fact that $\nabla^k \Psi$  has a compact support and the exponential decay of $Q_\omega$ to conclude the proof.
\end{proof}

\bibliographystyle{acm}
\bibliography{main.bbl}

\end{document}